\documentclass{article}
\usepackage{amsthm,amsbsy,amssymb,amsmath,amsfonts,txfonts}

\newtheorem{theorem}{Theorem}[section]
\newtheorem{proposition}[theorem]{Proposition}
\newtheorem{lemma}[theorem]{Lemma}

\newtheorem{claim}[theorem]{Claim}
\newtheorem{remark}[theorem]{Remark}
\newtheorem{definition}[theorem]{Definition}

\newtheorem{conjecture}{Conjecture}

\newcommand{\LP}{L_p}
\newcommand{\linf}{\ell_{\infty}}
\newcommand{\opLP}{\mathcal{L}(L_p)}
\newcommand{\opLOne}{\mathcal{L}(L_1)}
\newcommand{\X}{\mathcal{X}}
\newcommand{\Y}{\mathcal{Y}}
\newcommand{\Z}{\mathcal{Z}}
\newcommand{\opX}{\mathcal{L}(\X)}
\newcommand{\N}{\mathbb{N}}

\newcommand{\RN}{\mathbb{R}}
\newcommand{\LCN}{\lambda\in\mathbb{C}}
\newcommand{\HaarApendix}{\{h_{n,i}\}_{n=0,\, i=1}^{\infty\,\,\,\,\,\,\, 2^n}}
\newcommand{\HaarXX}{\{h_{0,0}\}\cup\{h_{n,i}\}_{n=0,\, i=1}^{\infty\,\,\,\,\,\,\, 2^n}}
\newcommand{\sumk}{\sum_{k=1}^{\infty}}
\newcommand{\seq}[1] {\displaystyle \{#1_i\}_{i=1}^{\infty}}
\newcommand {\dseq}[1] {\{#1_{n,i}\}_{n=0,\,i=1, }^{\infty,\,\,\,\,\, 2^n}}
\newcommand{\sumX}{\big (\sum \X\big )_p}
\newcommand{\IX}{\mathcal{M}_\X}
\newcommand{\IXone}{\mathcal{M}_{L_1}}
\newcommand{\IXp}{\mathcal{M}_{L_p}}
\newcommand {\indf}{\mathbf{1}_{\Delta_i}}

\newcommand{\AlgE}{\mathcal{E}}
\newcommand{\pbounds}{1 \leq {\lowercase {p}} <\infty}
\newcommand{\strictpbounds}{1 < {\lowercase {p}} <\infty}
\DeclareMathOperator{\supp}{supp}
\DeclareMathOperator{\vspan}{span}

\def\e{\varepsilon}

\title{Commutators on $\LP$, $1\leq p <\infty$\thanks{AMS Classification: 47B47, 46E30}}
\author{D. Dosev\thanks{Young Investigator, NSF Workshop in Analysis and Probability, Texas A\&M University}, W. B. Johnson\thanks{Supported in part by NSF DMS-1001321 and U.S.-Israel Binational Science Foundation}, G. Schechtman\thanks{Supported in part by U.S.-Israel Binational Science Foundation. Participant NSF Workshop in Analysis and Probability, Texas A\&M University}}

\begin{document}
\maketitle

\centerline{Dedicated to the memory of Nigel Kalton}
\begin{abstract}
The operators on $\LP=L_p[0,1]$, $1\leq p<\infty$, which are not commutators are those of the form $\lambda I + S$ where
$\lambda\neq 0$ and $S$ belongs to the largest ideal in $\opLP$. The proof involves new structural results for operators on $\LP$ which are of independent interest.
\end{abstract}


\section{Introduction}
When studying derivations on a general Banach algebra $\mathcal{A}$, a natural  problem that arises is to classify the commutators in the algebra; i.e.,  elements of the form $AB-BA$. The problem as stated is hard to tackle on general Banach algebras.  The only known obstruction was proved in 1947  by Wintner(\cite{Wintner}).  He  proved that the identity in a unital Banach algebra is not a commutator, which immediately implies that no operator of the form  $\lambda I + K$, where $K$  belongs to a norm closed (proper) ideal $\mathcal{I}$ of $\mathcal{A}$ and $\lambda\neq 0$, is a commutator in the Banach algebra $\mathcal{A}$.  On the other hand, there seems to be no general conditions for checking whether an element of a Banach algebra is a commutator.

The situation changes if instead of an arbitrary Banach algebra we consider the algebra $\opX$ of all bounded linear operators on the Banach space $\X$. In this setting, one hopes that the underlying structure of the space $\X$ will provide enough information about the operators on $\X$ to allow one to attack the problem successfully. Indeed, this is the case provided the space $\X$ has some ``nice'' properties. The first complete classification of the commutators in  $\opX$ was given in 1965 by Brown and Pearcy (\cite{BrownPearcy}) for the case $\X = \ell_2$. They   proved that the only operators in $\mathcal{L}(\ell_2)$ that are not commutators have the form $\lambda I + K$, where $K$  is compact and  $\lambda\neq 0$. In 1972, Apostol proved  in \cite{Apostol_lp} that the same classification holds for the commutators on $\ell_p$, $1<p<\infty$, and  one year later, he proved that the same classification holds in the case of $\X = c_0$ (\cite{Apostol_c0}). Apostol had some partial results in \cite{Apostol_lp} and \cite{Apostol_c0} about special classes of operators on $\ell_1$, $\linf$, and $C([0,1])$, but he was unable to obtain a complete classification of the commutators on any of those spaces. A year before Apostol's results, Schneeberger proved  that the compact operators on $\LP$, $1<p<\infty$, are commutators but, as it will become apparent later, one needs a stronger result in order to classify the commutators on these spaces.

All of the aforementioned spaces have one common property; namely,
if $\X = \ell_p$, $1\leq p<\infty$, or $\X = c_0$ then $\X\simeq \sumX$ ($p=0$ if $\X = c_0$). It turns out that this property  plays an important role for proving the classification of the commutators on other spaces.
Thirty five years after Apostol's result, the first author obtained in \cite{Dosev} a complete classification of the commutators on $\ell_1$, which, as one may expect, is the same as the classification of the commutators on $\ell_2$.
A common feature of all the spaces $\X=\ell_p$, $1\le p<\infty$ and $\X=c_0$ is that the ideal of compact operators ${\mathcal K}(\X)$ on $\X$ is the largest non-trivial ideal in $\opX$.  The situation for $\X=\ell_\infty$ is different. Recall that an operator $T:\X\to\Y$ is strictly singular provided the restriction of $T$ to any infinite dimensional subspace of $\X$ is not an isomorphism.  On $\ell_p$, $1\le p <\infty$, and on $c_o$, every strictly singular operators is compact, but on ${\mathcal L}(\ell_\infty)$, the ideal of strictly singular operators contains non-compact operators (and, incidentally, agrees with the ideal of weakly compact operators).  In ${\mathcal L}(\ell_\infty)$, the ideal of strictly singular operators is the largest ideal, and it was proved in  $\cite{DJ}$ that all operators on $\linf$ that are not commutators have the form $\lambda I + S$, where $\lambda\neq 0$ and $S$  is strictly singular.

The classification of the commutators on $\ell_p$, $1\leq p\le \infty$, and on $c_0$, as well as partial  results on other spaces, suggest the following:
\begin{conjecture}\label{conj}
Let $\X$ be a Banach space such that $\X\simeq \sumX$, $1\leq p\leq\infty$ or $p=0$
(we say that such a space admits a Pe\l czy\'nski decomposition). Assume that $\opX$ has a largest ideal
$\mathcal{M}$. Then every non-commutator on $\X$ has the form $\lambda I + K$, where $K\in \mathcal{M}$  and $\lambda\neq 0$.
\end{conjecture}
Here and elsewhere in this paper, when we refer to an ideal of operators we always mean a non-trivial, norm closed, two sided ideal.
This conjecture is stated in $\cite{DJ}$.
To verify  Conjecture \ref{conj} for a given Banach space  $\X$, one must prove two steps:

{\textbf{Step 1.}} Every operator $T\in\mathcal{M}$ is a commutator.\\
{\textbf{Step 2.}} If $T\in\opX$ is not of the form $\lambda I + K$, where $K\in\mathcal{M}$ and $\lambda\neq 0$, then $T$ is a commutator.

The methods for proving {\textbf{Step 1}} in most  cases where the complete classification of the commutators on the space $\X$ is known are based on the fact that if $T\in\mathcal{M}$ then for every subspace $Y\subseteq\X$,  $Y\simeq\X$ and every $\varepsilon > 0$ there exists a complemented subspace $Y_1\subseteq Y$, $Y_1\simeq Y$ such that $\|T_{|Y_1}\|<\varepsilon$.
 Let us just mention that this fact is fairly easy to see if $T$ is a compact operator on $c_0$ or $\ell_p$, $1\leq p<\infty$ (\cite[Lemma 9]{Dosev}, see also \cite{Apostol_lp}).  (Throughout this work,  $Y \simeq X$  means that $X$ and $Y$ are isomorphic; i.e., linearly homeomorphic; while
$Y \equiv X$ means that the spaces are isometrically isomorphic.)

Showing the second step is usually more difficult than showing {\textbf{Step 1}}. In most cases for which we have a complete characterization of the commutators on $\X$, we use the following theorem, which is an immediate consequence of Theorem 3.2 and Theorem 3.3 in \cite{DJ}.
\begin{theorem}\label{workthm}
Let $\X$ be a Banach space such that $\X\simeq\sumX$,  $1\leq p\leq\infty$ or $p=0$. Let $T\in\opX$ be such that there exists a subspace $X\subset\X$ such that $X\simeq \X$, $T_{|X}$ is an isomorphism, $X+T(X)$ is complemented in $\X$, and $d(X, T(X))>0$. Then $T$ is a commutator.
\end{theorem}
In the previous theorem the distance is defined as the distance from $Y$ to the unit sphere of $X$.
The basic idea is to prove that if $T\in\opX$ is not of the form $\lambda I + K$, where $K\in\mathcal{M}$ and $\lambda\neq 0$, then $T$ satisfies the assumptions of Theorem \ref{workthm} and hence $T$ is a commutator. This is not obvious even for the  classical sequence spaces $c_0$ and
$\ell_p$, $1\leq p<\infty$, but it suggests what one may try to prove for other classical Banach spaces in order to obtain a complete characterization of the commutators on those spaces.

Following the ideas in \cite{DJ}, for a given Banach space $\X$ we define the set
\begin{equation}\label{maxideal}
\IX = \{T\in\opX\,:\, I_{\X}\,\, \textrm{does not factor through}\,\, T \}.
\end{equation}
(We say that $S\in\opX$ factors trough $T\in\opX$ if there are $A,B\in\opX$ such that $S = ATB$.)
As noted in \cite{DJ}, this set comes naturally from the investigation of the structure of the commutators on several classical Banach spaces. In the cases of $\X = \ell_p$, $1\leq p\leq\infty$, and $\X = c_0$,
the set $\IX$ is the largest ideal in $\opX$ (observe that if $\IX$ is an ideal then it is the largest ideal in $\opX$ and $\IX$ is an ideal if and only if it is closed under addition). It is also known that $\IX$ is the largest ideal for $\X = \LP$, $\pbounds$, which we discuss later.

In some special cases of finite sums of Banach spaces we know that the classification of the commutators on the sum depends only on the classification of the commutators on each summand. In particular, this is the case with the space $\ell_{p_1}\oplus \ell_{p_2}\oplus\cdots\oplus \ell_{p_n}$ where the first two authors proved in \cite{DJ} that all non-commutators on $\ell_{p_1}\oplus \ell_{p_2}\oplus\cdots\oplus \ell_{p_n}$ have the
form $\lambda I + K$ where $\lambda\neq 0$ and $K$ belongs to some ideal in $\mathcal{L}(\ell_{p_1}\oplus \ell_{p_2}\oplus\cdots\oplus \ell_{p_n})$.

In this paper we always denote $L_p=L_p([0,1],\mu)$, where $\mu$ is the Lebesgue measure. Our main structural results are:
\begin{theorem}\label{movingthm}
Let $T\in\opLP$, $1\leq p<2$.
If $T-\lambda I\notin {\mathcal M}_{L_p}$ for all $\LCN$ then there exists a subspace $X\subset\LP$ such that $X\simeq \LP$, $T_{|X}$ is an isomorphism, $X+T(X)$ is complemented in $\LP$, and $d(X, T(X))>0$.
\end{theorem}

\begin{theorem}\label{compactrestrictionthm}
Let $T\in\opLP$, $1\leq p<2$. If $T\in  {\mathcal M}_{L_p}$ then for every $Y\subseteq\LP$, $Y\simeq\LP$, there exists a subspace $X\subset Y$ such that $X$ is complemented in $\LP$, $X\simeq \LP$, and  $T_{|X}$ is a compact operator.
\end{theorem}

Notice that Theorem \ref{compactrestrictionthm} implies that for $1\leq p <2$, $ {\mathcal M}_{L_p}$ is closed under addition and hence is the largest ideal in $\opLP$. It follows by duality that for $2<p<\infty$,  $ {\mathcal M}_{L_p}$ is closed under addition as well and hence is the largest ideal in $\opLP$. This duality argument is needed because Theorem \ref{compactrestrictionthm} is false for $p>2$. To see that Theorem \ref{compactrestrictionthm} is false for $p>2$ one can consider $T = JI_{p,2}$ where $I_{p,2}$ is the identity from $\LP$ into $L_2$ and $J$ is an isometric embedding  from $L_2$ into $\LP$.

In order to prove Theorem \ref{movingthm} for $1<p<2$, it was necessary to improve   \cite[Proposition 9.11]{JMST} for the spaces $L_p$, $1<p<2$, and the improvement is of independent interest.  In Theorem \ref{thm:stabilization} we show that for a natural equivalent norm on $L_p$, $1<p<2$, if $T$ is an operator on $L_p$ which is an isomorphism on a copy of $L_p$, then some multiple of $T$ is almost an isometry on an isometric copy of $L_p$.  The proof of Theorem \ref{thm:stabilization}, which can be read independently from the rest of this paper, is the most difficult argument in this paper and we will postpone it till the Appendix.

Using   Theorem \ref{movingthm} and Theorem \ref{compactrestrictionthm} it is easy to show that Conjecture \ref{conj} also holds for $\LP$, $1\leq p<2$. It follows by duality that  Conjecture \ref{conj} also holds for $\LP$, $2<p<\infty$.

\begin{theorem}
Let $\mathcal{M}$ be the largest ideal in $\opLP$, $1\leq p<\infty$. An operator $T\in\opLP$ is a commutator if and only if  $\,\,T-\lambda I\notin \mathcal{M}$ for any $\lambda\neq 0$.
\end{theorem}
\begin{proof}
As we already mention, we only need to consider the case $1\leq p<2$ and the case $2<p<\infty$ will follow by a duality argument.\\
If $T$ is a commutator, from the remarks we made in the introduction it follows that $T-\lambda I$ cannot be in $\mathcal{M}$ for any $\lambda\neq 0$. For proving the other direction we have to consider two cases:\\
{\bf{Case I.}} If $T\in \mathcal{M}$ ($\lambda = 0$), we first apply Theorem \ref{compactrestrictionthm} to obtain a complemented subspace $X\subset \LP$ such that $T_{|X}$ is a compact operator and then apply  \cite[Corollary 12]{Dosev} which gives us the desired result.\\
{\bf{Case II.}} If $\,\,T-\lambda I\notin \mathcal{M}$  for any $\LCN$ we are in position to apply Theorem \ref{movingthm}, which combined with Theorem \ref{workthm} imply that $T$ is a commutator.
\end{proof}

The rest of this paper is devoted to the proofs of Theorems \ref{movingthm} and \ref{compactrestrictionthm}. We consider the case $L_1$ separately since some of the ideas and methods used in this case are quite different from those used for the case  $L_p$, $1 < p<\infty$.


\section{Notation and basic results}
Throughout this manuscript, if $\X$ is a Banach space and $X\subseteq\X$ is complemented, by $P_X$ we denote a projection from $\X$ onto $X$.
For any two subspaces (possibly not closed) $X$ and $Y$  of a Banach space $\Z$ let
$$
d(X,Y) = \inf \{\|x-y\| : x\in S_{X},\, y\in Y\}.
$$
A well known consequence of the open mapping theorem is that for any two closed subspaces $X$ and $Y$ of $\Z$, if $X\cap Y = \{0\}$ then $X+Y$ is a closed subspace of $\Z$ if and only if $d(X,Y)>0$.
Note also that $2d(X,Y)\geq d(Y,X) \geq 1/2d(X,Y)$, thus $d(X,Y)$ and $d(Y,X)$ are equivalent up to a constant factor of $2$. The following proposition was proved in \cite{DJ} and will allow us later to consider only isomorphisms instead of arbitrary operators on $\LP$.

\begin{proposition}[{\cite[Proposition 2.1]{DJ}}]\label{distanceprop}
Let $\X$ be a Banach space and $T\in\opX$ be  such that there exists a
subspace $Y\subset\X$ for which $T$ is an isomorphism on $Y$ and $d(Y,TY)>0$. Then for every $\LCN$, $(T-\lambda I)_{|Y}$ is an isomorphism and $d(Y,(T-\lambda I)Y)>0$.
\end{proposition}

We will also need a result similar to Proposition \ref{distanceprop}, where instead of adding a multiple of the identity we want to add an arbitrary operator. Obviously that cannot be done in general, but if we assume that the operator we add has a sufficiently small norm we can derive the desired conclusion.
\begin{proposition}\label{smallperturbationProp}
Let $T\in\opX$ and let $Y\subset \X$ be such that $T$ is an isomorphism on $Y$, $Y\simeq \X$, $d(Y,TY)>0$, and $Y+TY$ is a complemented subspace  of $\X$ isomorphic to $\X$. Then there exists an $\varepsilon >0$, depending only on $d(Y,TY)$, the norm of the projection onto $Y+TY$, and  $\|T_{|Y}^{-1}\|$  such that if $K\in \opX$ satisfies $\|K_{|Y}\|<\varepsilon$ then $d(Y,(T+K)Y)>0$ and $Y+(T+K)Y$ is a complemented subspace  of $\X$ isomorphic to $\X$.
\end{proposition}
\begin{proof}
First we show that if $\varepsilon$ is sufficiently small then $d(Y,(T+K)Y)>0$, provided $\|K\|<\varepsilon$.
As in \cite[Proposition 2.1]{DJ}, we have to show that there exists a constant $c>0$ such that  for all $y \in S_Y$, $d((T+K)y,Y)>c$. From $d(TY,Y)>0$ it follows that there exists a constant $C$, such that for all $y \in S_Y$, $d(Ty,Y)>C$. If $\|K_{|Y}\|<\frac{C}{2}$ then
$$
\|(T+K)y-z\|\geq \|Ty-z\| - \|Ky\|\geq d(Ty,Y) - \frac{C}{2}\geq \frac{C}{2}
$$
for all $z\in Y$  hence $d((T+K)Y,Y)>0$. \\
Let $P$ be the projection onto $Y+TY$. To show that $Y+(T+K)Y$ is complemented in $\X$ we first define an isomorphism $S : Y+TY\to Y+(T+K)Y$ by $S(y+Tz) = y + (T+K)z$ for every $y,z\in Y$. From the definition of
$S$ we have that $\|S-I\|\leq C(Y,T)\|K_{|Y}\|$ (where $\displaystyle C(Y,T) = \frac{\|P\|\|T_{|Y}^{-1}\|}{d(TY,Y)}$), hence if $\|K_{|Y}\|$ is small enough the operator $R=SP+I-P$ is an isomorphism on $\X$. Now it is not hard to see that $RPR^{-1}$ is a projection onto $Y+(T+K)Y$.
\end{proof}


\section{Operators on $\LP$, $\strictpbounds$}

Recall (see (\ref {maxideal}) in the Introduction) that if $\X$ is a Banach space, $ \IX = \{T\in\opX\,:\, I_{\X}\,\, \textrm{does not factor through}\,\, T \}$, then $T \not\in  \IX $ if and only if there exists a subspace $X$ of $\X$ so that $T_{|X}$ is an isomorphism, $TX$ is complemented in $\X$, and $TX\simeq\X$.

As we have already mentioned in the Introduction, the set $\IX$ is the largest ideal in $\opX$ if and only if it is closed under addition. Using the fact that if $p=1$ then $\IXone$ coincides with the ideal of non-$E$ operators, defined in \cite{Enflo_Starbird}, and if $\strictpbounds$ then $\IXp$ coincides with the ideal of non-$A$ operators, defined in \cite{JMST}, it is clear that $\IX$ is in fact the largest ideal in those spaces. This fast, as we already mentioned, follows from Theorem \ref{compactrestrictionthm} as well. For more detailed discussion of the $E$ and $A$ operators we refer the reader to \cite{Enflo_Starbird} and \cite[Section 9]{JMST}  and let us also mention that  we are not going to use any of the properties of the $E$ or $A$  operators  and so do not repeat their definitions here.

In this section we mainly consider operators $T: \LP \to \LP$, $\pbounds$, that preserve a complemented copy of $\LP$, that is, there exists a complemented subspace  $X\subseteq\LP$, $X\simeq\LP$ such that $T_{|X}$ is an isomorphism. The fact that we can automatically take a complemented subspace isomorphic to $\LP$ instead of just a subspace isomorphic to $\LP$ follows from \cite[Theorem 9.1]{JMST} in the case $p>1$ and \cite[Theorem 1.1]{Rosenthal_L1} in the case $p=1$. From the definition of $\IX$ it is easy to see that $T\notin \IXp$ if and only if $T$ maps a copy of $\LP$ isomorphiclly onto a complemented copy of $L_p$.

Also, recall that an operator $T:\X\to \Y$ is called $Z$-strictly singular provided the restriction of $T$ to any subspace of $\X$, isomorphic to $Z$, is not an isomorphism. From the remarks above, it is clear that the class of operators from $\LP$ to $\LP$ that do not preserve a complemented copy of  $\LP$ coincides with the class of $\LP$-strictly singular operators, hence the class of $\LP$-strictly singular operators is the largest ideal in $\opLP$.
\begin{definition}\label{HaarBasis}
The sequence of functions $\HaarXX$ defined by $h_{0,0}(t)\equiv 1$ and, for
$n = 0, 1,\ldots$ and $i=1,2,\ldots, 2^n$,
$$
h_{n,i}(t) = \left\{ \begin{array}{ll}
1 & \textrm{if}\,\, t\in ((2i-2)2^{-(n+1)}, (2i-1)2^{-(n+1)})\\
-1
& \textrm{if}\,\,t\in ((2i-1)2^{-(n+1)}, 2i2^{-(n+1)})\\
0 & \textrm{otherwise}
\end{array} \right.
$$
is called the Haar system on $[0,1]$.
\end{definition}
The Haar system, in its natural order, is an unconditional monotone basis of $\LP [0,1]$
for every $\strictpbounds$ (cf. \cite[p.3, p.19]{LT}) and we denote by $C_p$ the
unconditional basis constant of the Haar system. As usual, by
$\{r_n\}_{n=0}^{\infty}$ we denote the Rademacher sequence on $[0,1]$, defined
by $r_n = \sum_{i=1}^{2^{n}}h_{n,i}$.
\begin{definition}\label{def:squarefunction}
Let $\{x_i\}_{i=1}^{\infty}$ be an unconditional basis for $\LP$.
For $x = \sum_{i=1}^{\infty} a_ix_i$, the square function of $x$ with respect to $\{x_i\}_{i=1}^{\infty}$ is defined by
$$
S(x)  = \left ( \sum_{i=1}^{\infty}a_i^2x_i^2\right )^{\frac{1}{2}}.
$$
\end{definition}

The following proposition is well known. 
We include its proof here for completeness.
\begin{proposition}\label{respsuppproj}
Let $\seq{a}$ be a block basis of the Haar basis for $\LP$, $\strictpbounds$, such that $A = \overline{\mathrm{span}} \{a_i\,:\, i=1,2,\ldots\}$ is a complemented subspace of $\LP$ via a projection $P$. Then there exists a projection onto $A$ that respects supports with respect to the Haar basis and whose norm depends on $p$ and $\|P\|$ only.
\end{proposition}
\begin{proof}
Define $\sigma_i = \{(k,l)\,:\,h_{k,l}\in \supp (a_i)\}$, where the support is taken with respect to the Haar basis, and denote $X_i = \overline{\mathrm{span}} \{h_{k,l}\, :\, (k,l)\in\sigma_i\}$. It is clear that all spaces $X_i$ are $C_p$  complemented in $\LP$, via the natural projections $P_i$, as a span of subsequence of the Haar basis.
Consider the operator $P_A = \sum_iP_iPP_i$. Provided it is bounded, it is easy to check that $P_A$ is a projection onto $A$ that respect supports. In order to show that $P_A$ is bounded consider the formal sum
$$
P = \sum_{i,j} P_iPP_j.
$$
A simple computation shows that
\begin{equation}
\|P_A\| = \|\mathbb{E}\sum_{i,j} \varepsilon_i\varepsilon_j P_iPP_j\|\leq
\mathbb{E}\|\sum_{i,j} \varepsilon_i\varepsilon_j P_iPP_j\|\leq C_p^2\|P\|,
\end{equation}
where $\varepsilon_i$ is a Rademacher sequence on $[0,1]$, which finishes the proof.
\end{proof}

The following theorem is the main result of this section. We will postpone its proof till the end since the ideas for proving it deviate from the general ideas of this section and the proof as well as the result are of  independent interest.
Recall \cite{Rosenthal_L1} that an operator $T$ on $L_p$, $1\leq p <\infty$, is called a sign embedding provided there is a set $S$ of positive measure and $\delta>0$ so that $\|Tf\| \ge \delta$ whenever $\int f \,d\mu =0$ and $|f|=1_S$ almost everywhere.


\begin{theorem} \label{thm:stabilization} For each $1<p<2$ there is a
constant $K_p$ such that if $T$ is a sign embedding operator from $L_p[0,1]$ into $L_p[0,1]$ (and in particular if it is an isomorphism), then there is a $K_p$ complemented subspace $X$ of $L_p[0,1]$ which is $K_p$-isomorphic to $L_p[0,1]$ and such that some multiple of $T_{|X}$ is a $K_p$-isomorphism and $T(X)$ is $K_p$ complemented in $L_p$. \hfill\break
Moreover, if we consider $L_p[0,1]$ with the norm $|\|x\||_p=\|S(x)\|_p$ (with $S$ being the square function with respect to the Haar system) then, for each $\varepsilon>0$, there is a subspace $X$ of $L_p[0,1]$ which is $(1+\varepsilon)$-isomorphic to $L_p[0,1]$ and such that some multiple of $T_{|X}$ is a $(1+\varepsilon)$-isomorphism (and $X$ and $T(X)$ are $K_p$ complemented in $L_p$).
\end{theorem}

\begin{remark}\label{rem:stabilization}
Note that Theorem \ref{thm:stabilization} is also true for $p=1$. This result follows from \cite[Theorem 1.2]{Rosenthal_L1}, where it is shown that
if $T\in\opLOne$ preserves a copy of $L_1$ then given $\varepsilon>0$, $X$ can be chosen isometric to $L_1$ so that some multiple of $T_{|X}$ is $1+\varepsilon$ isomorphism. Having that remark in mind, sometimes we may use Theorem \ref{thm:stabilization} for the case $p=1$ as well.
\end{remark}

Before we continue our study of the operators on $\LP$ that preserve a copy of $\LP$ we prove Theorem \ref{compactrestrictionthm} in the case of $\LP$, $1<p<2$. For this we need two lemmas for non sign embeddings and $\LP$-strictly singular  operators on $\LP$ that we use both in the next section and later on.


\subsection{$\LP$ - strictly singular operators}

Lemma \ref{compactlemma} was proved in \cite{Rosenthal_L1} for the case $p=1$, and basically the same proof works for general $p$, $1\leq p<\infty$.
\begin{lemma}\label{compactlemma}
Let $T\colon L_p\to L_p$, $\pbounds$ be a non sign embedding operator. Then for all subsets $S\subset\mathbb{R}$ with positive measure there exists a subspace $X\subset \LP(S)$ of $\LP$, $X\equiv \LP$,  such that $T_{|X}$ is compact.
\end{lemma}
\begin{proof}
We can choose by induction sets $A_i$ in $S$ such that $A_1 = S$, $A_{n} = A_{2n}\cup A_{2n+1}$, $A_{2n}\cap A_{2n+1}=\emptyset$, $\displaystyle \mu(A_{2n}) = \frac{1}{2}\mu(A_n)$, and
$\displaystyle \|Tx_n\|<\frac{\varepsilon}{2^{n+1}}$ where
$\displaystyle x_n = \frac{\mathbf{1}_{A_{2n}}  - \mathbf{1}_{A_{2n+1}}}{\mu(A_n)^{\frac{1}{p}}}$.
In order to do that assume that we have $A_1, A_2, \ldots , A_k$ where $k$ is an odd number and let $n = \frac{k+1}{2}$. Since $T$ is not a sign embedding, there exists $y_n$ such that $\int y_n = 0$,
$|y_n| = \mathbf{1}_{A_n}$, and $\|Ty_n\|\leq \frac{\varepsilon}{2^{n+1}}\mu(A_n)^{\frac{1}{p}}$. Then set $x_n = \frac{y_n}{\mu(A_n)^{\frac{1}{p}}}$,
$A_{2n} = \{x : x_n(x) = 1\}$, and $A_{2n+1} = \{x : x_n(x) = -1\}$. It is not hard to see that $(x_i)_{i=1}^{\infty}$ is isometrically equivalent to the usual sequence of  Haar functions and hence
$X =\mathrm{span} \{x_i\}$ is isometric to $L_p$. From the fact that $(x_i)$ is a monotone basis for $X$ it is easy to deduce that $T_{|X}$ is a compact operator.
\end{proof}
Note that Lemma \ref{compactlemma} immediately implies that for every complemented subspace $Y$ of $\LP$, $Y\simeq\LP$, there exists a complemented subspace $X\subseteq Y$,  $X\simeq \LP$,  such that $T_{|X}$ is compact.
The following lemma, which is also an immediate consequence of Lemma \ref{compactlemma} and Theorem \ref{thm:stabilization}, will be used in the sequel.
\begin{lemma}\label{smallnormLemma}
Let $T\in\opLP$, $1\leq p<2$,  be an $\LP$-strictly singular operator. Then for any $X\subseteq\LP$, $X\simeq\LP$  and  $\varepsilon >0$ there exists  $Y\subseteq X$, $Y\simeq\LP$ such that $\|T_{|Y}\|<\varepsilon$.
\end{lemma}
\begin{proof}
Again, this result immediately follows from Lemma \ref{compactlemma} and Theorem \ref{thm:stabilization}.
A $\LP$-strictly singular operator cannot be a sign embedding (Theorem \ref{thm:stabilization}) and then we use the construction in Lemma \ref{compactlemma}. From the fact that $(x_i)$ is a monotone basis for $X$ it follows that $\|T_{|X}\|<\varepsilon$.
\end{proof}
\begin{remark}
Clearly the proof we have for Lemma \ref{smallnormLemma} depends heavily on Theorem \ref{thm:stabilization} which is the deepest result of this paper. We do not know if an analogue of Lemma \ref{smallnormLemma} holds for $2<p<\infty$.
\end{remark}

\begin{proof}[Proof of Theorem \ref{compactrestrictionthm} in the case of $\LP$, $1\leq p<2$.]
This is a direct consequence of Theorem \ref{thm:stabilization} and Lemma \ref{compactlemma}.
First we find a complemented subspace $X'\subset Y$, $X'\simeq \LP$.
Now we observe that an $\LP$-strictly singular operator cannot be a sign embedding operator  (Theorem \ref{thm:stabilization}) and then we use Lemma \ref{compactlemma} for $X'$ to find a complemented subspace $X\subset X'$, $X\simeq\LP$, such that $T_{|X}$ is a compact operator.
\end{proof}


\subsection{Operators that preserve a copy of $\LP$}
\begin{definition}
For $T:\LP\to\LP$ and $X\subseteq\LP$ define the following two quantities:
\begin{align}
f(T,X) &= \inf_{x\in S_X} \|Tx\|\\
g(T,X) &= \sup_{\begin{array}{c} Y\subseteq X \\ Y\simeq\LP \end{array}}  f(T,Y).
\end{align}
\end{definition}
Clearly $f(T,X)$ does not decrease and $g(T,X)$ does not increase if we pass to subspaces. For an arbitrary subspace $Z\subset \LP$, $Z\simeq\LP$ note the following two (equivalent) basic facts:
\begin{itemize}
\item $T_{|Z'}$ is an isomorphism for some $Z'\subset Z$, $Z'\simeq\LP$ if and only if $g(T,Z)>0$
\item $T_{|Z}$ is $\LP$-strictly singular if and only if $g(T,Z)=0$
\end{itemize}
\begin{proposition}\label{PerturbationProp}
Let $S\colon\LP\to\LP$, $1\leq p<2$, be an $\LP$-strictly singular operator and let $Z$ be a subspace of $\LP$ which is also isomorphic to $\LP$. Then for every operator $T\in\opLP$ we have $g(T+S,Z) = g(T,Z)$.
\end{proposition}
\begin{proof}
If $T_{|Z}$ is $\LP$-strictly singular then $(T+S)_{|Z}$ is also $\LP$-strictly
 singular  hence
$g(T,Z)= g(T+S,Z)=0$. For the rest of the proof we consider the case where there exists $Z'\subset Z$, $Z'\simeq\LP$ such that $T_{|Z'}$ is an isomorphism hence $g(T,Z)>0$. \\
Let $0<\varepsilon < g(T,Z)/4$  and let $Y\subseteq Z$, $Y\simeq\LP$ be such that  $g(T,Z)-\varepsilon < f(T,Y)$. Using Lemma \ref{smallnormLemma} we find $Y_1\subseteq Y$, $Y_1\simeq\LP$ such that $\|S_{|Y_1}\|<\varepsilon$. Now
$$
g(T+S,Z)\geq f(T+S,Y_1)>f(T,Y_1)-\varepsilon \geq f(T,Y)-\varepsilon >g(T,Z)-2\varepsilon
$$
hence  $g(T+S,Z)\geq g(T,Z)-2\varepsilon$. Switching the roles of $T$ and $T+S$ (apply the previous argument for
$T+S$ and $-S$) gives us $g(T,Z)\geq g(T+S,Z)-2\varepsilon$ and since $\varepsilon$ was arbitrary small we conclude that $g(T+S,Z) = g(T,Z)$.
\end{proof}

\begin{lemma} \label{disjointifyLemma}
Let $X$ and $Y$ be two subspaces of $\LP$, $1\leq p < 2$, such that $X\simeq Y\simeq \LP$. Then there exist subspaces $X_1\subseteq X$, $Y_1\subseteq Y$ such that $X_1\simeq Y_1\simeq \LP$, $d(X_1, Y_1)>0$, and $X_1+Y_1$ is complemented subspace of $\LP$. Moreover, $X_1$ and $Y_1$ can be chosen in such a way that there exists projection onto $X_1+Y_1$ with norm depending only on $p$.
\end{lemma}
\begin{proof}
Without loss of generality, by passing to a subspace if necessary, we may assume that $X\simeq Y\simeq L_p$ are two complemented subspaces of $\LP$.
Our first step is to find two subspaces $X_1\subseteq X$ and $Y_1\subseteq Y$ which are isomorphic to $\LP$ and $d(X_1,Y_1)>0$.

Let $P\colon \LP\to X$ and $Q\colon \LP\to Y$ be two onto projections. We consider two cases for the operator $Q_{|X}\colon X\to Y$:

\textbf{Case 1.}  $Q_{|X}$ is $\LP$-strictly singular. Fix $\delta >0$. Using Lemma \ref{smallnormLemma}  we find $X_1\subseteq X$, $X_1\simeq L_1$ such that $\|Q_{|X_1}\|<\delta$. We are going to show that $d(X_1,Y)>0$.\\
Let $x\in S_{X_1}$ and $y\in Y$ be arbitrary. If $\displaystyle \|y\|\notin [\frac{1}{2},2]$ then clearly $\displaystyle \|x-y\|>\frac{1}{2}$. If not, then
$$
\|x-y\|\geq \frac{\|Qx-Qy\|}{\|Q\|} = \frac{\|Qx-y\|}{\|Q\|}\geq \frac{\|y\|-\|Qx\|}{\|Q\|}\geq \frac{1}{\|Q\|}(\frac{1}{2}-\delta).
$$
Since $x$ and $y$ were arbitrary we conclude that $d(X_1, Y)>0$ (here we take $Y_1 = Y$).

\textbf{Case 2.} $Q_{|X}$ is not $\LP$-strictly singular. Fix $\delta >0$. Then using Theorem \ref{thm:stabilization} we find $X'\subseteq X$ such that $X'$ is isomorphic to $\LP$ and $\lambda Q_{|X'}$ is a $K_p$ isomorphism for some $\lambda>0$. Without loss of generality we may assume that $X' = \LP (\nu)$ for some non-atomic measure $\nu$.
Now we find two disjoint, $\nu$-measurable sets $A$ and $B$ with positive measure and denote $X_1 = \LP(A)\subseteq \LP(\nu)$, $X_2 = \LP(B)\subseteq \LP(\nu)$, and  $Y_1 = QX_2$. Clearly $d(X_1, X_2) = 1$ and we are going to show that $d(X_1, QX_2)>0$.

Let $x\in S_{X_1}$ and $y=\lambda Qz\in QX_2$. Then
$$
\|x-y\|\geq \frac{\|Qx-Qz\|}{\|Q\|} =  \frac{\|\lambda Qx-\lambda Qz\|}{\lambda\|Q\|}\geq \frac{\|x-z\|}{K_p\lambda \|Q\|}\geq \frac{1}{K_p\lambda \|Q\|}
$$

Having $X_1$ and $Y_1$ from our first step, without loss of generality (by passing to a subspace if necessary) we may assume that $X_1$ and $Y_1$ are $K_p$ complemented in $\LP$ and $K_p$ isomorphic to $\LP$ and, for simplicity of notation, we will use $X$ and $Y$ instead of $X_1$ and $Y_1$.

Let $P\colon \LP\to X$ and $Q\colon \LP\to Y$ be two onto projections of norm
$K_p$. It is easy to see that $(I-Q)_{|X}$ is not a $\LP$-strictly singular operator. If we assume that this is not the case, fix $\delta >0$ and using Lemma \ref{smallnormLemma} we find $X'\subset X$, $X'\simeq \LP$, such that $\|(I-Q)_{|X'}\|<\delta$. But then for $x'\in S_{X'}$ and $Qx'\in Y$ we have $\|x'-Qx'\|<\delta$ which is a contradiction with $d(X,Y)>0$ since $\delta$ was arbitrary. Similarly, we show that $(I-P)_{|Y}$ is not a $\LP$-strictly singular operator.

Fix $\varepsilon >0$. Let $X_1\subset X$, $X_1\simeq \LP$, be such that $(I-Q)_{|X_1}$ is an isomorphism and some multiple of $(I-Q)_{|X_1}$ is a $K_p$ isomorphism. By Theorem \ref{thm:stabilization} there exists a $K_p$ complemented subspace $X_2\subset (I-Q)(X_1)$ which is also $K_p$ isomorphic to $\LP$.
Denote $X'={((I-Q)_{|X_2})}^{-1}(X_2)$. Now  $X'\simeq \LP$ and it is easy to see that $X'$ is complemented in $\LP$ (via the projection ${((I-Q)_{|X_2})}^{-1}P_{X_2}(I-Q)$ of norm at most $K_p^2$, where $P_{X_2}$ is a projection of norm $K_p$ onto $X_2$).

Similarly, we find $Y'\subseteq Y$ such that $Y'$ is $K_p$ isomorphic to $\LP$, $K_p^2$ complemented in $\LP$, and $(I-P)_{|Y'}$ is an isomorphism.

Let $R_1$ and $R_2$ be projections onto $(I-Q)X'$ and  $(I-P)Y'$, respectively,  of norm at most $K_p^2$. Denote by $V_1: (I-Q)X'\to X'$ the inverse map of $(I-Q)_{|X'}:X'\to (I-Q)X'$ and, similarly, denote by $V_2: (I-P)Y'\to Y'$ the inverse map of $(I-P)_{|Y'}:Y'\to (I-P)Y'$. Then a basic algebraic computation shows that $PV_1R_1(I-Q) + QV_2R_2(I-P)$ is a projection onto $X'+Y'$ of norm at most $2K_p^4$ which finishes the proof.
\end{proof}

\begin{lemma}\label{UNCLemma}
Let $T\in\opLP$, $1\leq p<2$, and assume that for every $X\subseteq\LP$, $X\simeq\LP$, there exists  a subspace  $X_1\subseteq X$, $X_1\simeq\LP$, and a constant $\lambda = \lambda(X_1)$ such that $T_{|X_1} = \lambda I_{X_1} + S$ where $S_{|X_1}$ is an $\LP$-strictly singular operator. Then there exists  a constant $\lambda$ and an $\LP$-strictly singular operator $S$, depending only on $T$, such that $T = \lambda I + S$.
\end{lemma}
\begin{proof}
Let $X$ and $Y$ be arbitrary subspaces of $\LP$ which are also isomorphic to $\LP$ and let $X_1\subseteq X$ and $Y_1\subseteq Y$ be the subspaces from the statement of the lemma. We will show that $\lambda(X_1) = \lambda(Y_1)$. Without loss of generality, using Lemma \ref{disjointifyLemma}, we can assume that $X_1\cap Y_1 = \{0\}$ and $X_1+Y_1$ is a closed and complemented subspace of $\LP$. Let
$$
T_{|X_1}  = \lambda_1 I_{X_1} + S_1\qquad , \qquad T_{|Y_1}  = \lambda_2 I_{Y_1} + S_2.
$$
Let $\tau\colon X_1\to Y_1$ be an isomorphism and define $Z = \{x+\tau (x)\,|\, x\in X_1\}$.
\begin{equation}\label{TrestZ}
T_{|Z} = T(x+\tau x) = \lambda_1x+\lambda_2\tau x + S_1x + S_2\tau x.
\end{equation}
The operator $S\colon Z\to \LP$ defined by $S(x+\tau x) = S_1x + S_2\tau x$ is $\LP$-strictly singular as a sum of two such operators. From the assumption of the lemma, there exist $Z_1\subseteq Z$ and $\lambda_3 \in\mathbb{C}$ such that
 $Z_1\simeq \LP$ and
\begin{equation}\label{TrestZ1}
T_{|Z_1} = \lambda_3 I_{Z_1} + S_3
\end{equation}
where $S_3$ is $\LP$-strictly singular. From (\ref{TrestZ}) and (\ref{TrestZ1}) we obtain that
the operator $T_1\colon Z\to \LP$ defined by
$T_1(x+\tau x) = \lambda_1x+\lambda_2\tau x - \lambda_3(x+\tau x)$ is also $\LP$-strictly singular on $Z_1$, i.e $T_1|_{Z_1}$ is $\LP$ strictly singular.
The last conclusion is possible if and only if $\lambda_1 = \lambda_2 = \lambda_3$. In fact, if we assume that $\lambda_1 \neq \lambda_3$, then, the operator $T_1$ will be an isomorphism on $Z$ because for every $x\in S_{X_1}$ we will have
$$
\|T_1(x+\tau x)\| = |\lambda_1-\lambda_3|\left \|x + \frac{\lambda_2-\lambda_3}{\lambda_1-\lambda_3}\tau x \right \| \geq |\lambda_1-\lambda_3|d(X_1, Y_1)\|x\|\geq \frac{|\lambda_1-\lambda_3|}{1+\|\tau\|}d(X_1, Y_1)\|x+\tau x\|.
$$
Let $\lambda = \lambda(X_1)$ for every subspace $X_1$ as in the statement of the lemma.
Now it easily follows that $\lambda I-T$ is $\LP$-strictly singular. Indeed, if we assume otherwise, then there exists a subspace $Z\subseteq\LP$, $Z\simeq\LP$ such that $(\lambda I-T)_{|Z}$ is an isomorphism. But according to the assumptions of the lemma, there exists $Z_1\subset Z$, $Z\simeq\LP$ such that $(\lambda I-T)_{|Z_1}$ is $\LP$-strictly singular which contradicts the fact that $(\lambda I-T)_{|Z}$ is an isomorphism. This finishes the proof.
\end{proof}

An immediate corollary of Lemma \ref{UNCLemma} is that for an operator $T\in\opLP$, $\pbounds$, not of the form $\lambda I + S$, where $S$ is an $\LP$-strictly singular operator, there exists a complemented subspace $X\subset\LP$, $X\simeq\LP$ such that  $(T-\lambda I)_{|X}$ preserves a copy of $\LP$ for every $\LCN$, and this is in fact what we are going to use in the sequel.

\begin{lemma}\label{MainLemma}
Let $T\in\opLP$, $1\leq p<2$, and assume that for every $X\subseteq \LP$, $X\simeq\LP$, and every $\varepsilon >0$ there exist  $X_1\subset X$, $X_1\simeq\LP$, and $\lambda = \lambda(X_1)$ such that $g(\lambda I - T, X_1)<\varepsilon$.
Then for every $\varepsilon>0$ there exists  $\lambda_\varepsilon$  such that $g(\lambda_\varepsilon I - T, \LP)<D_p\varepsilon$ where  $D_p$ is a constant depending only on $p$.
\end{lemma}
\begin{proof}
From the assumption in the statement of the lemma, without loss of generality we may assume that for every $X\subseteq \LP$, $X\simeq\LP$, and every $\varepsilon >0$ there exist  $X_1\subset X$, $X_1\simeq\LP$, and $\lambda = \lambda(X_1)$ such that $g(\lambda I - T, X_1)<\varepsilon$ and $(\lambda I - T)_{|X_1}$ is an
isomorphism. If $g(\lambda I - T, X_1)>0$ this can be achieved by passing to a subspace $Y_1$ of $X_1$ for which $f(\lambda I - T, Y_1)>0$.\\
If $g(\lambda I - T, X_1)=0$ for each $X_1\subset X$, $X_1\simeq\LP$, using the
fact that $g$ is a continuous function of $\lambda$, we find $\lambda_0$ such that $0<g(\lambda_0 I - T, X)<\varepsilon$ and then find $X_1\subset X$, $X_1\simeq\LP$ such that $f(\lambda I_0 - T, X_1)>0$.\\
Fix an $\varepsilon >0$ and let $Y_1$ and $Y_2$ be any two subspaces of $\LP$ such that $Y_1\simeq Y_2\simeq\LP$.
From our assumptions, there exist complemented subspaces $Y'_1, Y'_2$ such that $Y'_i\subseteq Y_i$, $Y'_i\simeq\LP$, and $g(\lambda(Y'_i) I - T, Y'_i)<\varepsilon$ for $i=1,2$ and without loss of generality, using Theorem \ref{thm:stabilization} and passing to a subspace if necessary, we may assume that $Y'_1$ and $Y'_2$ are $K_p$ complemented in $\LP$ and $((\lambda(Y'_i) I - T)_{|Y'_i}$ is a $K_p$-isomorphism for $i=1,2$. Then we apply Lemma \ref{disjointifyLemma} to get subspaces $X_1, X_2$ such that
\begin{itemize}
\item $X_i\subseteq Y'_i$, $X_i\simeq\LP$ for $i=1,2$
\item $X_1\cap X_2 = \{0\}$
\item $X_1$ and $X_2$ are $d_p$ complemented and $d_p$ isomorphic to $\LP$ via $\lambda(Y'_1) I - T$ and $\lambda(Y'_2) I - T$, respectively, for some constant $d_p$ depending only on $p$ (this follows from Lemma \ref{disjointifyLemma}  and our choice of $Y'_1$ and $Y'_2$)
\item $X_1+X_2$ is closed  and complemented subspace of $\LP$ and
there exists a projection onto $X_1+X_2$ with norm depending  only on $p$

\end{itemize}
 Since $X_i\subseteq Y'_i$, $i=1,2$, we have $g(\lambda(Y'_i) I - T, X_i)<\varepsilon$ for $i=1,2$ which in view of our choice of $X_1$ and $X_2$ implies
\begin{equation}\label{normbound}
\max (\|(\lambda(Y'_1)I-T)_{|X_1}\|, \|(\lambda(Y'_2)I-T)_{|X_2}\|) <d_p\varepsilon.
\end{equation}
Our goal is to show that $|\lambda(Y'_1)-\lambda(Y'_2)|<c_p\varepsilon$ for some constant $c_p$ independent of $Y'_1$ and $Y'_2$.

Let $\tau\colon X_1\to X_2$ be an isomorphism such that $\|\tau\|\leq d_p^2$ and $\|\tau^{-1}\|=1$. Define
$$
Z = \{ x+\tau x\,\,|\,x\in X_1\}.
$$
By assumption, there exists $Z'\subseteq Z$, $Z'\simeq\LP$, and  $\lambda(Z')$ such that  $0<g(\lambda(Z') I - T, Z')<\varepsilon$ and $\lambda(Z') I - T$ is an isomorphism on $Z'$ (like the argument in the beginning of the proof). Using Theorem \ref{thm:stabilization} we find $Z''\subseteq Z'$ such that $Z''$ is $K_p$ isomorphic to $\LP$ via $\lambda(Z') I - T$ (clearly $g(\lambda(Z') I - T, Z'')<\varepsilon$). Let $U = \lambda(Z') I - T$ and
define an operator $S\colon Z' \to \LP$ by $S(x+\tau x) = \lambda(Y'_1)x + \lambda(Y'_2)\tau x -Tx - T\tau x$.
A simple application of the triangle inequality combined with (\ref{normbound})  implies
\begin{equation}\label{eq:est1}
\|S(x+\tau x)\|\leq d_p^3\varepsilon (\|x\| + \|\tau x\|).
\end{equation}
From our choice of $Z''$ we also have
\begin{equation}\label{eq:est2}
\|\lambda (Z')(x+\tau x) - T(x+\tau x)\|\leq K_p\varepsilon(\|x\| + \|\tau x\|)
\end{equation}
and combining (\ref{eq:est1}) and (\ref{eq:est2}) gives us
\begin{equation}\label{eq:est3}
\|(U-S)(x+\tau x)\| \leq (K_p+d_p^3)\varepsilon (\|x\| + \|\tau x\|)\leq (K_p+d_p^3)(1+d_p^2)\varepsilon \|x\|.
\end{equation}
On the other hand
\begin{equation}\label{eq:est4}
\begin{split}
\|(U-S)(x+\tau x)\| &= \|(\lambda (Z')-\lambda (Y'_1))x + (\lambda (Z')-\lambda (Y'_2))\tau x\|\geq A_p (|\lambda (Z')-\lambda (Y'_1)|\|x\|+ |\lambda (Z')-\lambda (Y'_2)|\|\tau x\|)\\
&\geq A_p (|\lambda (Z')-\lambda (Y'_1)|+ |\lambda (Z')-\lambda (Y'_2)|)\|x\|) \geq A_p|\lambda (Y'_2)-\lambda (Y'_1)|\|x\|
\end{split}
\end{equation}
 ($A_p$ depends on $d_p, K_p$ and the norm of the projection onto $X_1+X_2$ which also depends on $p$ only). Combining (\ref{eq:est1}) and (\ref{eq:est2}) we get $|\lambda(Y'_1) - \lambda(Y'_2)|<\frac{(K_p+d_p^3)(1+d_p^2)}{A_p} \varepsilon$.
Now we define $\lambda_\varepsilon = \lambda (Y'_1)$ and it is not hard to check
that $g(\lambda_\varepsilon I - T, \LP)\leq (1+\frac{(K_p+d_p^3)(1+d_p^2)}{A_p}) \varepsilon$.
\end{proof}

\begin{remark} Let $\pbounds$.
Suppose that there exists a sequence of numbers $\seq{\lambda}$ such that
$g(\lambda_nI-T,\LP)\xrightarrow[n\to\infty]{} 0$. Then there exists $\lambda$ such that $g(\lambda I-T,\LP) = 0$. This is easy to see by noticing that $g(\lambda I-T,\LP)$ is bounded away from $0$ for large $\lambda$, so without loss of generality we can assume that the $\lambda_n\xrightarrow[n\to\infty]{} \lambda$. Then, if $g(\lambda I-T,\LP)=4\delta >0$, there exists  $Y\subseteq\LP$, $Y\simeq\LP$, such that $f(\lambda I-T,Y)>2\delta$. Now if $|\lambda - \mu|<\delta$, then $f(\mu I-T,Y)>\delta$  and hence $g(\mu I-T,\LP)>\delta$ which contradicts our original assumption about the sequence $\seq{\lambda}$.
\end{remark}

From the last remark it trivially follows that if $T\in\opLP$, $\pbounds$, is such that $\lambda I -T$ preserves a copy of $\LP$ for every  $\lambda$, then $\displaystyle \inf_{\lambda} g(\lambda I-T,\LP)>0$.

\begin{lemma}\label{auxlemma}
Let $T\in\opLP$, $1\leq p<2$, and assume that $\lambda I-T$ preserves a copy of $\LP$ for every $\LCN$. Then there exists $\varepsilon >0$ and a subspace $X\subseteq\LP$, $X\simeq\LP$, such that for every $X^{'}\subseteq X$, $X^{'}\simeq\LP$ , and every $\LCN$ we have $g(\lambda I-T, X^{'})>\varepsilon$.
\end{lemma}
\begin{proof}
The proof is a straightforward from Lemma \ref{MainLemma} and the last remark.
\end{proof}

The next result of this section is a reduction lemma that will help us later.
\begin{lemma}\label{reductionlemma}
Let $T,P \in\opLP$, $\pbounds$, be such that $P$ is a projection satisfying $P\LP\simeq\LP$ and  $(I-P)TP$ is an isomorphism on a subspace $X\subseteq\LP,\, X\simeq\LP$. Then there exists $Z\subset X$ such that $Z\simeq TZ\simeq\LP,\, d(TZ,Z)>0$, and $Z+TZ$ is a subspace isomorphic to $\LP$ and complemented in $\LP$.
\end{lemma}
\begin{proof}
Without loss of generality we may assume that $X\subseteq P\LP$ since $(I-P)T$ is an isomorphism on $PX$ and $PX\simeq (I-P)TPX\simeq X\simeq \LP$. Also without loss of generality, by passing to a subspace if necessary, we may assume that $X$ is complemented in $\LP$. \\
First we show that the conclusion of the lemma holds for
$T_1 = \frac{1}{2\|T\|}T + I$.  Clearly $T_1$ satisfies the assumptions of the lemma with the same subspace $X$ as in the statement. Since $T_1$ is an onto isomorphism, $T_1X$ is a complemented subspace of $\LP$ from which it is clear that $X+T_1X = T_1X$ is complemented.\\
To prove $d(X, T_1X)>0$ let $x\in S_X$ and $y\in X$. Then
$$
\|x-T_1y\|\geq \frac{\|(I-P)(x-T_1y)\|}{\|I-P\|} = \frac{\|(I-P)T_1y\|}{\|I-P\|} = \frac{\|(I-P)T_1Py\|}{\|I-P\|}\geq \frac{c}{\|I-P\|}\|y\|
$$
where $c$ is such that $\|(I-P)T_1Px\|\geq c\|x\|$ for all $x\in X$. If $\|T_1y\|\leq 1/2$ then $\|x-T_1y\|\geq 1/2$. Otherwise we have $\|y\| > \frac{1}{2\|T_1\|}$ and hence $\|x-T_1y\|>\frac{c}{2\|I-P\|\|T_1\|}$. From these estimates we can conclude that
\begin{equation}\label{eq:distX}
d(X,T_1X)\geq \max \left (\frac{1}{2}, \frac{c}{2\|I-P\|\|T_1\|}\right )>0.
\end{equation}
Now using Proposition \ref{distanceprop} for $T_1$ and $X$ we obtain that
$\frac{1}{2\|T\|}T$ is an isomorphism on $X$ and $d(\frac{1}{2\|T\|}TX,X)>0$, or equivalently, $T$ is an isomorphism on $X$ and $d(TX,X)>0$.
Finally, $X+TX = X+\frac{1}{2\|T\|}TX = X+(\frac{1}{2\|T\|}TX+I)X = X+T_1X$ hence $X+TX$ is complemented.
\end{proof}


 \begin{proof}[Proof of Theorem \ref{movingthm} for $L_p$, $1<p<2$ ]
In view of Lemma \ref{UNCLemma}, we can apply Lemma \ref{auxlemma} for $T$ and
let $X$ and $\varepsilon$ be the one from Lemma \ref{auxlemma}. Without loss of generality we may assume that $X$ is complemented (otherwise we may pass to a complemented subspace). Let $V$ be an isomorphism from $\LP$ into $\LP$ such that $V\LP = X$. \\
Fix $\delta>0$ to be chosen later. We will build sequences $\seq{y}$ in $X$, and   $\seq{a}$ and $\seq{b}$ in $\LP$ such that:
\begin{enumerate}
\item $\seq{a}$ and $\seq{b}$ are block bases of $\HaarXX$ such that if we denote $\sigma_i = \supp\{a_i\}\cup\supp\{b_i\}$, where the support is with respect to the Haar basis, then $\seq{\sigma}$ is a disjoint sequence of subsets of  $\HaarXX$
\item $\seq{y}$ is equivalent to the Haar basis for $\LP$
\item $\|y_i-a_i\|<\frac{\delta}{2^{i}}\|y_i\|$ and $\|Ty_i-b_i\|<\frac{\delta}{2^{i}}\|Ty_i\|$ for all  $1\leq i<\infty$.
\end{enumerate}
The construction of these sequences is similar the construction of Lemma \ref{disjointifyLemma} and we sketch it below for completeness. As before, by $P_{(k,s)}$ we denote the projection onto the linear span of $\{h_{n,i}\}_{n=k,\, i=1}^{s\,\,\,\,\,\,\,\, 2^n}$.

Let $y_1 = Sh_{1,1}$. There exists $n_1$ such that $\|y_1 - P_{(1,n_1)}y_1\|<\frac{\delta}{2}\|y_1\|$
and $\|Ty_1 - P_{(1,n_1)}Ty_1\|<\frac{\delta}{2}\|Ty_1\|$. Let $a_1 = P_{(1,n_1)}y_1$ and  $b_1 = P_{(1,n_1)}Ty_1$.
Denote $A_1^{+} = \{x: V^{-1}y_1(x) = 1\}$ and $A_1^{-} = \{x: V^{-1}y_1(x) = -1\}$. Let $\seq{z}$ be the Rademacher sequence on $A_1^{+}$ and $\seq{z'}$ be the Rademacher sequence on $A_1^{-}$. Using the fact  that the Rademacher sequence is weakly null, we find $l$ such that
\begin{equation*}
\begin{split}
\|Sz_l-P_{(n_1,\infty)}Sz_l\|&<\frac{\delta}{16}\|Sz_l\|\\
\|TSz_l-P_{(n_1,\infty)}TSz_l\|&<\frac{\delta}{16}\|TSz_l\|\\
\|Sz'_l-P_{(n_1,\infty)}Sz'_l\|&<\frac{\delta}{16}\|Sz'_l\|\\
\|TSz'_l-P_{(n_1,\infty)}TSz'_l\|&<\frac{\delta}{16}\|TSz'_l\|.
\end{split}
\end{equation*}
Define $y_2 = Sz_l$, $y_3 = Sz'_l$ and find $n_2$ such that $\|y_k-P_{(n_1,n_2)}y_k\|<\frac{\delta}{8}\|y_k\|$  and
$\|Ty_k-P_{(n_1,n_2)}Ty_k\|<\frac{\delta}{8}\|Ty_k\|$ for $k=2,3$. As before, let $a_k = P_{(n_1,n_2)}y_k$ and $b_k = P_{(n_1,n_2)}Ty_k$ for $k = 2,3$.\\
Continuing this way, we build the other elements of the sequences $\seq{y}$, $\seq{a}$ and $\seq{b}$.
From the construction it is clear that $\{(\supp y_i\cup\supp Ty_i)\}_{i=1}^{\infty}$ are  essentially disjoint. If we denote $Y = \overline{\mathrm{span}} \{y_i\,:\, i = 1,2,\ldots \}$ we have that $Y$ is a complemented subspace of $\LP$ which is also isomorphic to $\LP$. To see this it is enough to notice that $V^{-1}Y$ is isometric to $\LP$ (since it is spanned by a sequence which is isometrically equivalent to the Haar basis) and hence complemented in $\LP$. One projection onto $Y$ is given by $P_Y = VP_{V^{-1}Y}V^{-1}P_X$. From now on, without loss of generality we assume that $X=Y$ (since we can pass to a subspace in the beginning if necessary).\\
As  in the argument in Lemma \ref{disjointifyLemma}, using the principle of small perturbations, it is easy to see
that the subspace $A = \overline{\mathrm{span}} \{a_i\,:\, i = 1,2,\ldots \}$ is complemented. A projection onto $A$ is given by $P'_A = GP_YG^{-1}$ where $G\in\opLP$ is defined by
$$
G = I - \sum_iy_i^{*}(x)(y_i - a_i).
$$
Using Proposition \ref{respsuppproj} we find a projection $P_A$ onto $A$ that respect supports.
Let $S : A\to \LP$ be the operator defined by
$$
S = I - \sum_ia_i^{*}(x)(a_i - y_i)
$$
and note that  $SA = Y$($\equiv X$) and $\|I-S\|<4\delta$.
Let also $S' : A\to \LP$ be the operator defined by $S'a_i = Ta_i-b_i$ and let $T' = T - S'P_A$.
It follows easily that $\|S'\|<C\delta\|T\|$, where $C$ depends only on $p$, since
$$
\|S'a_i\| = \|Ta_i-b_i\| = \|T(a_i - y_i) + Ty_i-b_i\| < \frac{2\delta \|T\|\|y_i\|}{2^{i}}.
$$
First we show that $(I-P_A)T'P_A$ preserves a copy of $\LP$.\\
If not, we have that  $(I-P_A)T'P_A$ is $\LP$-strictly singular and hence $P_AT'P_A$ preserves a copy of $\LP$ (otherwise $T'_{|A}$ will be $\LP$-strictly singular which is false). We also have the inequality $g(P_AT'P_A, A)>\varepsilon/2$. In order to show it we need to go back to the definitions of $f(\cdot, \cdot)$ and $g(\cdot, \cdot)$. Fix $A'\subseteq A$ and $\LCN$. We shall show that $\displaystyle g(\lambda I+T', A')>\frac{\varepsilon}{2}$. We may assume that $|\lambda|<\|T'\|+1$(otherwise $f(\lambda I-T',A')>1$). Let $x\in S_{A'}$ be arbitrary.
\begin{equation*}
\begin{split}
\|(\lambda I - T')x\|& \geq \|(\lambda I - T')Sx\|  - \|(\lambda I - T')(x-Sx)\| \geq
\|(\lambda I - T)Sx\| -  \|S'P_ASx\| - \|(\lambda I - T')(x-Sx)\|\\
& \geq (1-4\delta)f((\lambda I - T),A') - C\delta\|T\|\|P_A\|(1+4\delta) - 4(2\|T'\|+1)\delta.
\end{split}
\end{equation*}
Taking  infimum over the left side we obtain $f((\lambda I - T'),A')>f((\lambda I - T),A')/2$ for sufficiently small $\delta$ and hence $g(\lambda I - T',A')>\varepsilon/2$ for every $\LCN$.
Using the fact that $(I-P_A)TP_A$ is an $\LP$-strictly singular operator and Proposition \ref{PerturbationProp} we obtain
$$
g(P_AT'P_A, A) + g((I-P_A)T'P_A,A) = g(T'P_A,A)>\frac{\varepsilon}{2}.
$$
Let $P_AT'a_i = \lambda_ia_i$ (we can do that since $P_A$ respect supports). Till the end of this proof it will be convenient to switch the enumeration of the $\seq{a}$ to $\dseq{a}$, which is actually how we constructed them.
For each $n$, using the pigeon-hole principle, we can find a set $\sigma(n)$ with cardinality at least
$\displaystyle \frac{\varepsilon 2^n}{4\|P_AT\|} $  and a number $\mu_n$ such that  $\displaystyle |\mu_n  - \lambda_i|<\frac{\varepsilon}{4}$ for every  $i\in\sigma(n)$. Clearly, there exists an infinite subset $N_1\subseteq\N$  and a number $\mu_\e$ such that
$$
\sum_{n\in N_1}|\mu_n - \mu_\e| < \frac{\varepsilon}{100}.
$$
Let $Z = \overline{\mathrm{span}} \{a_{n,i}\, :\, n\in N_1, i\in\sigma(n)\}$.
Using a result of Gamlen and Gaudet (see \cite{Gamlen_Gaudet}), we have that $Z\simeq\LP$ and clearly $Z$ is complemented in $\LP$.
Now note that
\begin{equation}\label{eq:c1}
\|(\mu_\e I-P_AT')_{|Z}\| < \frac{\varepsilon}{2}\,\,\mathrm{hence}\,\,g(\mu_\e I-P_AT',Z)< \frac{\varepsilon}{2}.
\end{equation}
On the other hand,
\begin{equation}\label{eq:c2}
\frac{\varepsilon}{2} < g(\mu_\e I-T',A) = g((\mu_\e I-P_AT') - (I-P_A)T',A) = g(\mu_\e I-P_AT',A)
\end{equation}
since $(I-P_A)T'_{|A}$ is $\LP$-strictly singular. The equations (\ref{eq:c1}) and (\ref{eq:c2}) lead to contradiction which shows that $(I-P_A)T'P_A$ preserves a copy of $\LP$, say $Z'$. Now $\|(I-P_A)T'P_A - (I-P_A)TP_A\| = \|(I-P_A)S'P_A\|<C\delta \|T\|(\|P_A\|+1)^2$ hence, for sufficiently small $\delta$, we have that $(I-P_A)TP_A$ is
an isomorphism on $Z'$ and clearly $d((I-P_A)TP_AZ',Z') > 0$. In view of Lemma \ref{reductionlemma} this finishes the proof.
\end{proof}


\section{Operators on $L_1$}

Recall that we have already proved Theorem \ref{compactrestrictionthm} in the case of $L_1$. The proof in this case does not involve anything new and can be done only using the ideas found in  \cite {Rosenthal_L1}, which we have already mentioned.

Now we switch   attention to the operators not of the form $\lambda I + K$ where $\LCN$ and $K$ is in ${\mathcal M}_{L_1}$.  Our investigation will rely on the representation Kalton gave for a general operator on $L_1$ in \cite{KaltonEnd}, but again Rosenthal's paper  \cite{Rosenthal_L1} is a better reference for us. Before we state Kalton's representation we need a few definitions.

\begin{definition}
An operator $T: L_1\to L_1$ is called an atom if $T$ maps disjoint functions to disjoint functions. That is,
if $\mu(\supp f\cap\supp g)=0$ then $\mu(\supp Tf\cap\supp Tg)=0$.
\end{definition}

Unlike the notation for $\strictpbounds$, here $\supp$ refers to the support with respect to the interval $[0,1]$. A simple characterization of the atoms is given by the following known structural result.
\begin{proposition}[{\cite[Proposition 1.3]{Rosenthal_L1}}]
 An operator $T: L_1\to L_1$ is an atom if and only if there exist measurable functions
$a : (0,1)\to \RN$ and $\sigma : (0,1)\to\ (0,1)$ with $Tf(x) = a(x)f(\sigma x)$ a.e. for all $f\in L_1$.
\end{proposition}

In \cite{Rosenthal_L1} the following definition is given:
\begin{definition}
Let $T : L_1\to L_1$ be a given operator.\\
(a) Say that $T$ has {\it atomic part} if  there exists a non-zero atom $A : L_1\to L_1$ with $0\leq A\leq |T|$. \\
(b) Say that $T$ is {\it purely continuous} if $T$ has no atomic part.\\
(c) Say that $T$ is {\it purely atomic} if $T$ is a strong $\ell_1$-sum of atoms.
\end{definition}
The condition (c) in the preceding definition simply means that there is a sequence of atoms $\{T_j\}_{j=1}^{\infty}$ from $L_1$ to $L_1$ and $K<\infty$ so that for all $f\in L_1$, $\sum\|T_jf\|\leq K\|f\|$ and $Tf = \sum T_jf$.

Here is  Kalton's representation theorem for operators on $L_1$ the way it is stated in \cite{Rosenthal_L1}.
\begin{theorem}
Let $T : L_1\to L_1$ be a given operator. There are unique operators $T_a, T_c\in\opLOne$ so that $T_a$ is purely atomic, $T_c$ is purely continuous, and $T = T_a + T_c$. Moreover, there exists a sequence of atoms
$\{T_j\}_{j=1}^{\infty}$ so that $T_a$ is a strong $\ell_1$-sum of $\{T_j\}_{j=1}^{\infty}$ and the following four conditions hold
\begin{enumerate}
\item $\sum_{i=1}^{\infty} \|T_if\|\leq \|T_a||\|f\|$
\item $(T_if)(x) = a_i(x)f(\sigma_ix)$ a.e. where $a_i : (0,1)\to \RN$ are measurable functions,  and $\sigma_i : (0,1)\to (0,1)$
\item For all $i\neq j$, $\sigma_i(x)\neq \sigma_j(x)$ a.e.
\item $|a_j(x)|\geq |a_{j+1}(x)|$ a.e.
\end{enumerate}
\end{theorem}
Note that if $E$ is a set of positive measure such that $a_j(x)\neq 0$ a.e on $E$, then $\mu(\sigma_j(E))>0$. Indeed, let $F\subset E$ be such that $|a_j(x)|>\alpha>0$ for every $x\in F$.
Now $T_j1_{\sigma_j(F)} = 1 _Fa_j$ implies that $\|T_j1_{\sigma_j(F)}\|>0$ hence $\mu(\sigma_j(F))>0$.
The power of Kalton's representation theorem is that it reduces many problems about operators on $L_1$ to measure theoretic considerations. This is illustrated in the proof of the following proposition.

\begin{proposition}\label{atomprop}
Let $T\in\opLOne$ be a non-zero atom such that $T\neq\lambda I$ for any $\LCN$. Then there exists a subspace $Y\subset L_1$ such that $Y\equiv L_1$, $d(Y,TY)>0$, and $Y+TY$ is complemented in $L_1$.
\end{proposition}
\begin{proof}
By the definition of   atom we have that $(Tf)(x) = a(x)f(\sigma x)$ for some $a$ and $\sigma$. We consider two possibilities depending on $\sigma$.

\textbf{1.} If $\sigma = id$ a.e on $(0,1)$ then $a(x)\not\nequiv \textrm{const}$ a.e  (otherwise $T = \lambda I$ for some $\lambda$). Then we find two different numbers $\lambda_1,\lambda_2$ and a positive number $\delta$ such that
\begin{itemize}
\item $|\lambda_1 - \lambda_2|>3\delta$
\item There are closed sets $\Delta_i = \{x : |a(x) - \lambda_i|<\delta\}$ so that $\mu(\Delta_i)>0$ for $i=1,2$
\end{itemize}
To see this we can consider a good enough approximation of $a(x)$ with a step function and without loss of generality we may assume that $\lambda_1\neq -1$. Note also that we can choose $\delta$ as small as we want (independent of $\lambda_1$ and $\lambda_2$) which choice we leave for later.
Clearly, $\Delta_1\cap\Delta_2 = \emptyset$ and, since they are closed, by shrinking $\delta$ we can assume that they are at a distance of at least $\delta$ apart. From our choice of
$\Delta_i$, $i=1,2$, we also have
$$
\|(Tf - \lambda_i f)\indf\| < \delta\|f\indf\|\,\,, i=1,2.
$$
Let $S : L_1(\Delta_1)\to L_1(\Delta_2)$ be an isometry and define
$$
Z = \{f_1 + Sf_1 : f_1\in L_1(\Delta_1)\}.
$$
Since $\|f_1+Sf_1\| = 2\|f_1\|$ we immediately have that $Z\equiv L_1$. To show that $d(Z,TZ)>0$ assume that
$\|T(g+Sg)\| = 1$ for some $g\in L_1(\Delta_1)$.
Then for arbitrary $f\in L_1(\Delta_1)$ we have
\begin{equation*}
\begin{split}
\|f+Sf - Tg - TSg\| &= \|f+Sf - ag - aSg\| = \|f-ag\| + \|Sf - aSg\|  \\
& = \|f-\lambda_1g + (\lambda_1-a)g\| + \|Sf - \lambda_2Sg + (\lambda_2-a)Sg\|\\
&\geq  \|f-\lambda_1g\| - \|(\lambda_1-a)g\| + \|Sf - \lambda_2Sg\| - \|(\lambda_2-a)Sg\|\\
&\geq \|f-\lambda_1g\| + \|f - \lambda_2g\| - \delta \|g\|  - \delta \|Sg\|\geq
|\lambda_1-\lambda_2|\|g\| - 2\delta\|g\|\geq \frac{|\lambda_1-\lambda_2|}{3} \|g\|.
\end{split}
\end{equation*}
Now we observe that $\|T(g+Sg)\| = 1$ implies $\|g\|\geq \frac{1}{2\|T\|}$ hence
$$
d(TZ,Z) = \inf_{\begin{array}{c} \|T(g+Sg)\| = 1   \\f,g\in L_1(\Delta_1)  \end{array} }
\|f+Sf - Tg - TSg\|\geq \frac{|\lambda_1-\lambda_2|}{6\|T\|}.
$$
Define $T_1f(x) = \lambda_1f(x)\mathbf{1}_{\Delta_1}(x) + \lambda_2f(x)\mathbf{1}_{\Delta_2}(x)$ and let $K = T-T_1$.
Denote by $P_1$ the natural, norm one, projection  from $L_1$ onto $L_1(\Delta_1)$ and let  $P = P_1+\frac{\lambda_2+1}{\lambda_1+1}SP_1$. It is easy to see that $P$ is an idempotent operator since $P_1SP_1\equiv 0$. To see that $P$ is a projection onto $Z+T_1Z$ note that
$$
Pf = P_1f+\frac{\lambda_2+1}{\lambda_1+1}SP_1f = \frac{1}{\lambda_1+1}((\lambda_1+1)P_1f +  (\lambda_2+1)SP_1f ) =
\frac{1}{\lambda_1+1}(P_1f+SP_1f + \lambda_1P_1f + \lambda_2SP_1f)\in Z+T_1Z
$$
Now we observe that $Z+TZ = Z+(T_1+ K)Z$ and use the fact that $\|K_{|Z}\|<\delta$ to conclude that for sufficiently small $\delta$, Proposition  \ref{smallperturbationProp} guarantees that the subspace $Z+TZ$ is complemented.

\textbf{2.} If $\sigma \neq id$ a.e on $(0,1)$ let $A = \{x\,\, |\,\, \sigma(x) = x\}$ and denote $A' = (0,1)\backslash A$ and $B = A'\cap \{x\,\,|\,\,a(x) \neq 0\}$. We have two cases depending on $\mu(B)$.

Case 1. $\mu(B)>0$\\
In this case we show that there exists $\Delta\subset B$ such that
$\mu(\Delta\cap\sigma(\Delta)) = 0$.
Denote $\alpha_k = \{x : |x-\sigma(x)|>\frac{1}{k}\}\cap B$. Obviously
$\cup_{k=1}^{\infty} \alpha_k = \{x : |x-\sigma(x)|>0\}\cap B = B$ and the latter set has positive measure by assumption,
hence there exists $k_0$ for which $\mu(\alpha_{k_0})>0$. Now
$$
\alpha_{k_0}  = \bigcup_{n=0}^{2k_0-1} \left (\alpha_{k_0}\cap\left[\frac{n}{2k_0},\frac{n+1}{2k_0}\right] \right ),
$$
so there exists $n_0$ such that if we denote
$\beta = \alpha_{k_0}\cap\left[\frac{n_0}{2k_0},\frac{n_0+1}{2k_0}\right]$ then  $\mu(\beta)>0$.
From the way we defined $\beta$ it is evident that $\beta\cap\sigma(\beta) = \emptyset$ because
$\textrm{diam}(\beta)<\frac{1}{2k_0}$ and $|x-\sigma(x)|>\frac{1}{k_0}$ for every $x\in\beta$.
It is also clear that $d(L_1(\beta),TL_1(\beta)) = 1$ since $L_1(\beta)$ and $TL_1(\beta)$ have disjoint supports. The fact that
$L_1(\beta)+TL_1(\beta)$ is complemented in $L_1$ follows from the facts that  $L_1(\beta)$ is norm-one complemented, $T$ is an isomorphism (hence $TL_1(\beta)$ is complemented), and $L_1(\beta)$ and $TL_1(\beta)$ have disjoint supports.

Case 2. $\mu(B) = 0$.\\
In this case we have $\mu(A)>0$ (otherwise $T$ will be a zero atom). There are two sub-cases:
\begin{itemize}
\item  If $a(x)\neq \mathrm{const}$ a.e on $A$.\\ Then we proceed as in the case $\sigma = id$ a.e on $(0,1)$ but we consider $A$ instead of $(0,1)$.
\item If $a(x) = \mathrm{const} = \lambda$ a.e on $A$.\\
Then again we proceed as in the case $\sigma = id$ a.e on $(0,1)$ considering $\lambda_1 = \lambda$ and $\lambda_2 = 0$. We can do this since $\mu(A')>0$.
\end{itemize}
\end{proof}
\begin{remark}
Note that Proposition \ref{atomprop} is also valid when considering operators
$T: L_1(\nu_1)\to L_1(\nu_2)$ where $\nu_1$ and $\nu_2$ are two non-atomic measures on some sub $\sigma$-algebra on $(0,1)$ and this is in fact how we are going to use it.
\end{remark}

Now we proceed to the proof of Theorem \ref{movingthm} in the case $p=1$.


\begin{proof}[Proof of Theorem \ref{movingthm} in the case $L_1$]
First, using Lemma \ref{auxlemma}, we find $\varepsilon>0$ and a subspace $X\subseteq L_1$ such that for every $X'\subseteq X$, $X'\simeq L_1$, and every $\LCN$ we have  $g(\lambda I - T,X')>\varepsilon$. Then consider a similarity $S$ on $L_1$ such that $SX = L_1(\Delta)$ where $\Delta$ can be any nonempty open interval. For $T' = STS^{-1}$ we have $g(\lambda I - T',X')>\varepsilon '$
for every $X'\subseteq SX = L_1(\Delta)$ where $\varepsilon ' =\frac{\varepsilon}{\|S\|\|S^{-1}\|} $. It is clear that it is enough to prove the theorem for $T'$ so without loss of generality we may assume that $T' = T$ and $\varepsilon = \varepsilon '$. \\
Let $T = T_a+T_c$ be the Kalton representation for $T$ and fix $\delta$.
First, we use Rosenthal's remark before Lemma 2.1 in \cite{Rosenthal_L1} to find an atom $V$ and set $\Delta_1\subseteq\Delta$ such that
$\displaystyle \|{T_a}_{|L_1(\Delta_1)}-V\|<\frac{\varepsilon}{10}$.
Since completely continuous operators are not sign embeddings, we  apply \cite[Lemma 3.1]{Rosenthal_L1} to find a norm one complemented subspace $X'\subseteq L_1(\Delta_1)$ such that $X'\equiv L_1$ and $\displaystyle \|{T_c}_{|X'}\|<\frac{\varepsilon}{10}$. From our choice of $X'$ it follows that
$$
g(\lambda I - V,X'')>\frac{\varepsilon}{2} \,\, \textrm{for every}\,\, \LCN \,\,\textrm{and every}\,\, X''\subseteq X', X''\simeq L_1 .
$$
From the last inequality it is clear that $V: X'\to L_1$ is a non-zero atom and $V_{|X'}\neq \lambda I$.  Now Proposition \ref{atomprop} gives us the desired result.
\end{proof}

\newpage
\nocite{*}

\section{Appendix}
Before we start with the proof of Theorem \ref{thm:stabilization} we recall some of the notation we previously used and note some of the properties of $S(x)$, the square function defined with respect to the Haar basis (for a general definition of $S(x)$ see Definition \ref{def:squarefunction}).\\
Recall that unless otherwise noted, $\LP$ denotes $L_p([0,1],\mu)$ where $\mu$ is the Lebesgue measure.
The unconditional basis constant of the usual Haar basis  $\HaarApendix$ in $L_p$, $1<p<\infty$, is denoted by $C_p$. Recall also that $\{r_n\}_{n=0}^{\infty}$
is the Rademacher sequence on $[0,1]$ (defined by  $r_n = \sum_{i=1}^{2^{n}}h_{n,i}$).

Denote by $\AlgE_n$ the finite algebra generated by the dyadic intervals $[(i-1)2^{-n},i2^{-n}]$, $i=1,2,\ldots 2^n$, and by $\AlgE$, the union of all these algebras. It is clear that the algebra $\AlgE_n$ is generated by the supports of $\{(h_{n,i})\}_{i=1}^{2^n}$.

If $f$ and $g$ are functions in $\LP$ which have disjoint supports with respect
to the Haar basis, then it is obvious that $S^2(f+g) = S^2(f)+S^2(g)$. This will be used numerous times. Let $\{x_k\}_{k=1}^{\infty}$ be a sequence of functions in $\LP$, $\strictpbounds$, which are disjointly supported  with respect to the Haar basis. Using the unconditionality of the Haar basis and Khintchine's inequality we obtain
\begin{equation}\label{eq:sqfuncineq1}
\left \|\sumk x_k \right\|_p \geq C_p^{-1}\left(\int_0^1\left\|\sumk r_k(u)x_k\right\|_p^p\,du\right)^{1/p}=
C_p^{-1}\left\|\left(\int_0^1 \left|\sumk r_k(u)x_k\right|^p\,du\right)^{1/p}\right\|_p
\geq
C_p^{-1}A_p\left\|\left(\sumk |x_k|^2\right )^{\frac{1}{2}}\right\|_p ,
\end{equation}
where $A_p$ is the constant from Khintchine's inequality.
If $x_k = \sum_{i}\alpha_{k,i} h_{n_{k,i}}$, $k=1,2,\ldots $ are disjointly supported vectors with respect to the Haar basis, using (\ref{eq:sqfuncineq1}) for
$\{\alpha_{k,i}h_{n_{k,i}}\}_{k,i}$   we obtain
\begin{equation}\label{eq:sqfuncineq}
\left \|\sumk x_k \right\|_p \geq C_p^{-1}A_p\left(\int_0^1 \left(\sumk S^2(x_k)\right )^{\frac{p}{2}}\right)^\frac{1}{p}.
\end{equation}
In a similar manner
\begin{equation}\label{eq:sqfuncineq2}
\left \|\sumk x_k \right\|_p \leq C_pB_p\left(\int_0^1 \left(\sumk S^2(x_k)\right )^{\frac{p}{2}}\right)^\frac{1}{p}.
\end{equation}
From the last two inequalities it follows that the norm $|\|\cdot\||_p=\|S(\cdot)\|_p$ is equivalent to the usual norm in $\LP$, $\strictpbounds$.
Now we proceed to the main theorem of this section.

\bf{Theorem \ref{thm:stabilization}.}\, \it{For each $1<p<2$ there is a
constant $K_p$ such that if $T$ is a sign embedding operator from $L_p[0,1]$ into $L_p[0,1]$ (and in particular if it is an isomorphism), then there is a $K_p$ complemented subspace $X$ of $L_p[0,1]$ which is $K_p$-isomorphic to $L_p[0,1]$ and such that some multiple of $T_{|X}$ is a $K_p$-isomorphism and $T(X)$ is $K_p$ complemented in $L_p$. \hfill\break
Moreover, if we consider $L_p[0,1]$ with the norm $|\|x\||_p=\|S(x)\|_p$ (with $S$ being the square function with respect to the Haar system) then, for each $\varepsilon>0$, there is a subspace $X$ of $L_p[0,1]$ which is $(1+\varepsilon)$-isomorphic to $L_p[0,1]$ and such that some multiple of $T_{|X}$ is a $(1+\varepsilon)$-isomorphism (and $X$ and $T(X)$ are $K_p$ complemented in $L_p$).}

\begin{proof}[Proof of Theorem \ref{thm:stabilization}]:
Let $T$ be as in the statement of the theorem. Without loss of generality (see e.g. Lemma 9.10 in \cite{JMST}
and note that only the boundedness of $T$ is used) $\{Th_{n,i}\}$ is a block basis of $\HaarApendix$. For $E\in \mathcal{E}_n$ put
\[
v_n(E)=S\left(\sum_{h_{n,i}\subseteq E}Th_{n,i}\right),
\]
where $h_{n,i}\subseteq E$ is a shorthand notation  for $\rm{supp}(h_{n,i})\subseteq E$. Put  also $v_n=v_n([0,1]) = S(\sum_{i=1}^{2^n}Th_{n,i}) = S(Tr_n).$

\begin{claim}\label{cl:equi-integrability}
The convex hull of $\{v_n^2\}$ is $p/2$-equi-integrable; i.e., the
set \[V=\left\{\left(\sum\alpha_n^2v_n^2\right)^{p/2}\ : \sum\alpha_n^2\le 1\right\}\]
is equi-integrable.
\end{claim}

\begin{proof} The proof is a refinement of the argument on page 265 of
\cite{JMST}. Since the convex hull of any finite set in
$L_{p/2}$ is $p/2$-equi-integrable, it follows that if the convex
hull of $\{v_n^2\}$ is not $p/2$-equi-integrable then there are
$\varepsilon_0>0$, successive subsets $\sigma_m\subset \N$, and disjoint subsets $\{A_m\}_{m=1}^\infty$ of $[0,1]$ such that for
$w_m^2=\sum_{n\in\sigma_m}\alpha_n^2v_n^2$, where
$\sum_{n\in\sigma_m}\alpha_n^2=1$ for all $m$, we have
\begin{equation}\label{eq:tmp1}
\left(\int_{A_m}w_m^p\right)^{1/p}\ge\varepsilon_0.
\end{equation}
Using (\ref{eq:sqfuncineq}) and (\ref{eq:tmp1}) (the estimate in (\ref{eq:sqfuncineq}) we can use since $\{Tr_n\}$ are disjointly supported with respect to the Haar basis), for all $\{a_m\}_{m=1}^\infty\in\ell_2$ we have
\[
\begin{array}{rl}
(\sum_{m=1}^\infty a_m^2)^{1/2}& =\|\sum_{m=1}^\infty
a_m\sum_{n\in \sigma_m}\alpha_nr_n\|_2
\ge\|\sum_{m=1}^\infty a_m\sum_{n\in \sigma_m}\alpha_nr_n\|_p \geq
\|T\|^{-1}\|\sum_{m=1}^\infty a_m\sum_{n\in \sigma_m}\alpha_nTr_n\|_p \\
&\ge \|T\|^{-1}C_p^{-1}A_p(\int_0^1(\sum_{m=1}^\infty a_m^2\sum_{n\in
\sigma_m}\alpha_n^2S^2(Tr_n))^{p/2})^{1/p} \\
&= \|T\|^{-1}C_p^{-1}A_p(\int_0^1(\sum_{m=1}^\infty a_m^2\sum_{n\in
\sigma_m}\alpha_n^2v_n^2)^{p/2})^{1/p} =
\|T\|^{-1}C_p^{-1}A_p(\int_0^1(\sum_{m=1}^\infty a_m^2w_m^2)^{p/2})^{1/p}\\
&\ge \|T\|^{-1}C_p^{-1}A_p(\sum_{m=1}^\infty
|a_m|^p\int_{A_m}w_m^p)^{1/p} \ge
\|T\|^{-1}C_p^{-1}A_p\varepsilon_0(\sum_{m=1}^\infty |a_m|^p)^{1/p}
\end{array}
\]
which leads us to contradiction since $p<2$.
\end{proof}

\begin{proposition}\label{pr:Lambda}
There is an additive $L_{p/2}^+$ valued measure, $\Lambda$, on $\mathcal{E}$ and there are successive convex combinations $u_m(\cdot)$ of $\{v_n^2(\cdot)\}$ such that for all $E\in\mathcal{E}$ we have $u_m(E)\to \Lambda(E)$ almost surely and in $L_{p/2}$. Moreover, for any sequence $\varepsilon_n\to 0$, there are measurable sets $D_n\subset [0,1]$ with $\mu(D_n)>1-\varepsilon_n$ and such that
 for all $E\in\mathcal{E}$ and all $n$, $u_m(E){\bf 1}_{D_n}\to \Lambda(E){\bf 1}_{D_n}$ as $m\to\infty$ also in $L_1$.
\end{proposition}

\begin{proof}
We start as in the proof of Lemma 6.4 in \cite{JMST}: Since the set
$V$ from Claim \ref{cl:equi-integrability} is bounded in $L_{p/2}$, by a result of Nikishin \cite{ni} for
each $\varepsilon>0$ there exists a set $D=D_\varepsilon\subset[0,1]$
of measure larger that $1-\varepsilon$ such that
\[
\sup_{v\in V}\int_D v d\mu<\infty.
\]
Note that we may assume that $D_{1/n}\subset D_{1/(n+1)}$ for $n=2,3,\dots$. As in the proof of \cite[Lemma 6.4]{JMST}, using the weak compactness of $V_{|D_{1/2}}\subset L_1$, we can find successive convex
combinations $u_m(\cdot)$ of the $v_n^2(\cdot)$ such that
$u_m(E){\bf 1}_{D_{1/2}}$ converges pointwise and in $L_1$ to
$\Lambda_1(E){\bf 1}_{D_{1/2}}$ for every $E\in \mathcal{E}$, where
$\Lambda_1{\bf
1}_{D_{1/2}}$ is $L_1^+$-valued additive measure. Now we can find successive convex combinations $w_m(\cdot)$ of the $u_m(\cdot)$ such that
$w_m(E){\bf 1}_{D_{1/3}}$ converges pointwise and in $L_1$ to
$\Lambda_2(E){\bf 1}_{D_{1/3}}$ for every $E\in \mathcal{E}$, where
$\Lambda_2{\bf
1}_{D_{1/3}}$ is $L_1^+$-valued additive measure.
Note that $\Lambda_2{\bf
1}_{D_{1/2}}=\Lambda_1{\bf
1}_{D_{1/2}}$. Continuing in this manner and taking a diagonal sequence of the sequences of successive convex combinations we get a sequence, which we still denote $u_m$, of successive convex combinations of the $v_n^2$ and a $L_0^+$-valued additive measure $\Lambda$ such that for every  $n$ $\Lambda{\bf
1}_{D_{1/n}}$ is $L_1^+$-valued and $u_m(E){\bf 1}_{D_{1/n}}$ converges, as $m\to\infty$, pointwise and in $L_1$ to $\Lambda(E){\bf 1}_{D_{1/n}}$ for every $E\in \mathcal{E}$.

It remains to show that the convergence is also in $L_{p/2}$ (on the whole interval). Since for each $E$, $\{u_m(E)\}$ is $p/2$-equi-integrable, it follows
that, given any $\delta>0$, if $n$ is large enough
$\int_{D_n^c}u_m(E)^{p/2}<\delta$ for all $m$. Consequently,
we also have $\int_{D_n^c}\Lambda(E)^{p/2}\le\delta$ and
\begin{equation*}
\begin{split}
\limsup_{m\to\infty} \int|u_m(E)-\Lambda(E)|^{p/2}
&\le \limsup_{m\to\infty} \int_{D_n}|u_m(E)-\Lambda(E)|^{p/2}+2\delta\\
&\le \limsup_{m\to\infty}\|(u_m(E)-\Lambda(E)){\bf
1}_{D_n}\|_1^{p/2}+2\delta=2\delta.
\end{split}
\end{equation*}
Since this is true for any $\delta$ we get the desired result.
\end{proof}

Note first that if we denote $C = (C_p^2B_pA_p^{-1}\|T\|)^p$ then for all $m$ and all $E\in \mathcal{E}$ we have $\int u_m(E)^{p/2}\le C\mu(E)$, where $\{u_m\}$ are from Proposition \ref{pr:Lambda}. Indeed,
let $u_m = \sum_{k\in\sigma_m}\alpha_{m,k}v_k^2$ where  $\{\sigma_m\}_{m=0}^{\infty}$ are successive subsets of $\mathbb{N}$ and $\{\alpha_{m,k}\}_{m=0,}^{\infty}\phantom{}_{k\in\sigma_m}^{}$  is sequence of non-negative numbers such that $\sum_{k\in\sigma_{m}}\alpha_{m,k} = 1$.
Then using (\ref{eq:sqfuncineq}), (\ref{eq:sqfuncineq2}) and the unconditionality of the Haar basis  we get
\[
\begin{array}{rl}
\int u_m(E)^{p/2}  &=  \int (\sum_{k\in\sigma_m}\alpha_{m,k}v_k^2(E))^{p/2} =
\int (\sum_{k\in\sigma_m}\alpha_{m,k}S^2(\sum_{h_{k,i}\subseteq E} Th_{k,i}))^{p/2}\\
&\leq (C_pA_p^{-1})^p\|\sum_{k\in\sigma_m}\alpha_{m,k}^{1/2}(\sum_{h_{k,i}\subseteq E}Th_{k,i})\|_p^p\leq (C_pA_p^{-1}\|T\|)^p\|\sum_{k\in\sigma_m}\alpha_{m,k}^{1/2}(\sum_{h_{k,i}\subseteq E}h_{k,i})\|_p^p \\
&\leq (C_p^2B_pA_p^{-1}\|T\|)^p\|(\sum_{k\in\sigma_m}\alpha_{m,k}S^2(\sum_{h_{k,i}\subseteq E}h_{k,i}))^{1/2}\|_p^p \leq (C_p^2B_pA_p^{-1}\|T\|)^p\|{\bf 1}_{E}\|_p^p = (C_p^2B_pA_p^{-1}\|T\|)^p\mu(E)\\
\end{array}
\]
This implies that for all $E\in \mathcal{E}$, $\int\Lambda(E)^{p/2}\le
C\mu(E)$. Now consider the linear operator $T$ defined on the functions of
the form $f=\sum_{i=1}^r a_i {\bf 1}_{E_i}$ by $Tf=\sum_{i=1}^r a_i
\Lambda(E_i)$, where the $E_i$-s are disjoint sets in $\mathcal{E}$.  Then $T$ is bounded as an operator from a subspace of $L_{p/2}$ to $L_{p/2}$. Indeed,
\[
\int|Tf|^{p/2}\le \int\sum_{i=1}^r |a_i|^{p/2}
\Lambda(E_i)^{p/2}\le C\int\sum_{i=1}^r |a_i|^{p/2} \mu(E_i)=
C\|f\|_{p/2}^{p/2}.
\]
Consequently, $T$ can be extended to all of $L_{p/2}$ and then we define
$\Lambda(E)=T{\bf 1}_E$ for all $E$ in the Borel $\sigma$-algebra.

\begin{remark}
From the comments above it follows that $\Lambda$ can be extended to a
$L_{p/2}^+$-valued measure satisfying $\int\Lambda(E)^{p/2}\le
C\mu(E)$ for some constant $C<\infty$ and for all $E$ in the Borel
$\sigma$-algebra. For each $n$, $\Lambda(E){\bf 1}_{D_n}$ is an
$L_1^+$-valued measure. Note also that, since $T$ is a sign
embedding, $\Lambda$ is not identically zero.
\end{remark}

\begin{lemma}\label{lm:stabilization}
Let $\Lambda$ be a
non zero $L_{p/2}^+$-valued measure on the
Borel $\sigma$-algebra $\mathcal{B}$ satisfying
$\int\Lambda(A)^{p/2}\le C\mu(A)$ for some $C<\infty$ and all
$A\in\mathcal{B}$. Then for all $\varepsilon>0$  there
 exist a set $A_0$ and a number $c$, $0<c\le C$, such that
$\int\Lambda(A_0)>0$ and
\[c\mu(A)\le\int\Lambda(A)^{p/2}\le c(1+\varepsilon)\mu(A)\]
for all $A\subseteq A_0$.
\end{lemma}

\begin{proof}
Fix an $\varepsilon>0$ and denote $m=\sup\{\int\Lambda(A)^{p/2}/\mu(A) \ ;\ A\in\mathcal{B}\}$. Let $B_0\in\mathcal{B}$ be such that
$\frac{\int\Lambda(B_0)^{p/2}}{\mu(B_0)}\ge \frac{m}{1+\varepsilon}$. Let also $\mathcal{C}$ be a maximal collection of disjoint Borel subsets of $B_0$ of positive measure satisfying
$\frac{\int\Lambda(B)^{p/2}}{\mu(B)} < \frac{m}{1+\varepsilon}$.
The collection $\mathcal{C}$ is necessarily countable and if we assume that $A_0=B_0\setminus\bigcup_{B\in\mathcal{C}}B$  has measure $0$ then we have
\[\begin{array}{rl}
\frac{m}{1+\varepsilon}\mu(B_0)&\le\int\Lambda(B_0)^{p/2}=\int(\sum_{B\in \mathcal{C}}\Lambda(B))^{p/2}\\
&\le \int\sum_{B\in\mathcal{C}}\Lambda(B)^{p/2}
<\frac{m}{1+\varepsilon}\sum_{B\in\mathcal{C}}\mu(B)\\
&= \frac{m}{1+\varepsilon}\mu(B_0)
\end{array}
\]
which is a contradiction. Therefore, $A_0$ satisfies the conclusion of the lemma with $c=\frac{m}{1+\varepsilon}$.
\end{proof}

\begin{lemma}\label{lm:eq_with_max_function}
Let $\Lambda$ be a $L_{p/2}^+$-valued measure and suppose that $A_0$ is such that
for all $A\subseteq A_0$ and some constant $c>0$,
\[c\mu(A)\le\int\Lambda(A)^{p/2}\le c(1+\varepsilon)\mu(A).\]
Then for any measurable partition $A_0=\cup_{i=1}^nF_i$,
\[
\int\max_{1\le i\le n}\Lambda(F_i)^{p/2}\ge
(1+\varepsilon)^{-p/(2-p)}c\mu(A_0)\ge
(1+\varepsilon)^{-2/(2-p)}\int\Lambda(A_0)^{p/2}.
\]
\end{lemma}

\begin{proof}
\[
\begin{array}{rl}
c\mu(A_0)&\le\int\Lambda(A_0)^{p/2}=\int(\sum_{i=1}^n\Lambda(F_i))^{p/2}\\
&\le \int(\sum_{i=1}^n\Lambda(F_i)^{p/2})^{p/2}\max_{1\le i\le
n}\Lambda(F_i)^{(1-p/2)p/2}\\
&\le (\int\sum_{i=1}^n\Lambda(F_i)^{p/2})^{p/2} (\int\max_{1\le
i\le
n}\Lambda(F_i)^{p/2})^{(1-p/2)}\\
&\le (1+\varepsilon)^{p/2}c^{p/2}(\mu(A_0))^{p/2}(\int\max_{1\le
i\le n}\Lambda(F_i)^{p/2})^{(1-p/2)}.
\end{array}
\]
Consequently,
\[
\int\max_{1\le i\le n}\Lambda(F_i)^{p/2}\ge
(1+\varepsilon)^{-p/(2-p)}c\mu(A_0)\ge
(1+\varepsilon)^{-2/(2-p)}\int\Lambda(A_0)^{p/2}.
\]
\end{proof}

Fix an $\e>0$ and let $A_0\in \mathcal{B}$ and $c$ be as in Lemma
\ref{lm:eq_with_max_function} so that for any partition
$A_0=\cup_{i=1}^nF_i$,
\begin{equation}\label{eq:A_0}
\int\max_{1\le i\le n}\Lambda(F_i)^{p/2}>
(1+\varepsilon)^{-p/(2-p)}c\mu(A_0).
\end{equation}
Approximating by a set from $\mathcal{E}$, we may assume that the
set $A_0$ satisfying (\ref{eq:A_0}) is in $\mathcal{E}$. Let
$\{E_{n,i}\}_{n=0}^\infty\phantom{}_{i=1}^{2^n}$ be a dyadic tree
of sets in $\mathcal{E}$ with $E_{0,1}=A_0$ and let
\[
M_n=\max_{1\le i\le 2^n}\Lambda(E_{n,i}).
\]
$M_n$ is a non increasing sequence of functions in $L_{p/2}^+$.
Denote by $M$ its limit (in $L_{p/2}$ or, equivalently, almost
everywhere). Clearly,
\[
\int M^{p/2}\ge (1+\varepsilon)^{-p/(2-p)}c\mu(A_0).
\]
We now define a sequence of functions $\varphi_n:[0,1]\to A_0$ which will play a role similar to the one played by the sequence with the same name in \cite[Lemma 9.8]{JMST}. For each $n$ order the set $\{1,2,\dots,2^n\}$ according to the order of the leftmost points in $\{\overline{E_{n,i}}\}$; i.e., $i\prec j$ if $\min\{t\in\overline{E_{n,i}}\} < \min\{t\in\overline{E_{n,j}}\}$. Let $\varphi_n:[0,1]\to A_0$ be defined by $\varphi_n(t)=\min\{t\in\overline{E_{n,i}}\}$ if $1\le i\le 2^n$ is the first, in the order $\prec$, such that $\Lambda(E_{n,i})(t)\ge M(t)$. For each $t$, $\{\varphi_n(t)\}$ is a non-decreasing and thus a converging sequence. Let $\varphi(t)$ denote its limit. Notice that
\[
{\bf 1}_{\varphi^{-1}(A)}(t)M(t)\le\Lambda(A)(t)
\]
for every $t$ and every $A$ which is a union of the interiors of the $E_{n,i}$-s. Indeed, it is enough to prove this for $A=E_{n,i}^\circ$ for some $n$ and $i$. But if $t\in\varphi^{-1}(E_{n,i}^\circ)$ then for $k$ large enough $E_{k,i(k)}\subset E_{n,i}$ (where $\varphi_k(t)$ is the leftmost point of $\overline{E_{k,i(k)}}$) and
\[
\Lambda(E_{n,i})(t)\ge \Lambda(E_{k,i(k)})(t)\ge M(t).
\]

Fix a sequence $\{\varepsilon_n\}_{n=1}^{\infty}$, $\e_n\to 0$, to be chosen later. Consider the vector measure
\[
m(A)=(\mu(A),\int_{\varphi^{-1}(A)}M^{p/2}),\ \ A\subseteq A_0
\]
and notice that it is non-atomic and even absolutely
continuous with respect to Lebesgue measure. Indeed, for $A$ in
the algebra generated by the $E_{n,i}$-s,
\[
\int_{\varphi^{-1}(A)}M^{p/2}\le
\int_{\varphi^{-1}(A)}\Lambda(A)^{p/2}\le \int\Lambda(A)^{p/2}\le
C\mu(A).
\]
The inequality clearly extends to all $A\in\mathcal{B}$,
$A\subseteq A_0$.

By Lyapunov's theorem one can find a partition of $A_0$ into two sets, $\tilde F_{1,1}$ and $\tilde F_{1,2}$, of equal $m$ measure. For any $\e_1>0$, we can perturb $\tilde F_{1,1}$ and $\tilde F_{1,2}$ slightly to get $F_{1,1}, F_{1,2}$ in the algebra generated by the $E_{n,i}$-s which satisfy
\[
\mu(F_{1,1})=\mu(F_{1,2})=\frac12\mu(A_0)
\]and
\[
\frac{(1-\e_1)}{2}\int
M^{p/2}\le\int_{\varphi^{-1}(F_{1,1})}M^{p/2},\int_{\varphi^{-1}(F_{1,2})}M^{p/2}\le
\frac{(1+\e_1)}{2}\int M^{p/2}.
\]
Now we partition each of $F_{1,1}$ and $F_{1,2}$ in a similar manner and then continue the process. This way, for every positive sequence $\{\varepsilon_n\}_{n=1}^{\infty}$, $\e_n\downarrow 0$, we construct a dyadic tree $\{F_{n,i}\}_{n=0,}^\infty\phantom{}_{i=1}^{2^n}$ of subsets of
$A_0$ such that the elements of the tree $\{F_{n,i}\}_{n=0}^\infty\phantom{}_{i=1,}^{2^n}$ are in the algebra generated by the $E_{n,i}$-s and for all $n=0,1,\dots,\ i=1,\dots,2^n$, we have
\begin{equation}\label{eq:tree}
\mu(F_{n,i})=2^{-n}\mu(A_0)
\end{equation}
and
\begin{equation}\label{eq:M}
2^{-n}\prod_{j=1}^n(1-\e_j)\int
M^{p/2}\le\int_{\varphi^{-1}(F_{n,i})}M^{p/2}\le
2^{-n}\prod_{j=1}^n(1+\e_j)\int M^{p/2}.
\end{equation}

Define $G_{n,i} = \varphi^{-1}(F_{n,i})$ for $n=0,1,\ldots$ and $i=1,2,\ldots,2^n$.
Fix  $1-\delta =(\prod_{j=1}^{\infty}(1-\e_j))(1+\varepsilon)^{-2/(2-p)}$. The main remaining ingredient in the proof of the theorem is the
following claim.

\begin{claim}\label{claim:main}
There exist $a$ and $b$,  $0<a<b = (1-\delta)^{-1}a$, such that  for all $N$ and all coefficients $\{a_{n,i}\}_{n=0}^N\phantom{}_{i=1,}^{2^n}$,
\[
a\|S(\sum_{n=0}^N\sum_{i=1}^{2^n} a_{n,i}h_{n,i})\|_p^p\le
\|\sum_{n=0}^N\sum_{i=1}^{2^n}
a_{n,i}^2\Lambda(F_{n,i})\|_{p/2}^{p/2}\le b
\|S(\sum_{n=0}^N\sum_{i=1}^{2^n} a_{n,i}h_{n,i})\|_p^p.
\]
\end{claim}

\begin{proof}
 We shall use the shorthand notation $h_{n,i}\supseteq h_{N,j}$
for ${\rm supp}(h_{n,i})\supseteq {\rm supp}(h_{N,j})$. The first
equality below follows by expressing $\Lambda(F_{n,i})$ in terms
of the $\Lambda(F_{N,j})$-s and changing the order of summation.
We also use Lemma \ref{lm:stabilization} and (\ref{eq:tree}).
\begin{equation*}
\begin{array}{rl}
\|\sum_{n=0}^N\sum_{i=1}^n
a_{n,i}^2\Lambda(F_{n,i})\|_{p/2}^{p/2}&=
\int(\sum_{j=1}^{2^N}\Lambda(F_{N,j})\sum_{(n,i);h_{n,i}\supseteq
h_{N,j}}a_{n,i}^2)^{p/2}\\
&\le
\int\sum_{j=1}^{2^N}\Lambda(F_{N,j})^{p/2}(\sum_{(n,i);h_{n,i}\supseteq
h_{N,j}}a_{n,i}^2)^{p/2}\\
&\le
c(1+\e)\mu(A_0)2^{-N}\int\sum_{j=1}^{2^N}(\sum_{(n,i);h_{n,i}\supseteq
h_{N,j}}a_{n,i}^2)^{p/2}\\
&=c(1+\e)\mu(A_0)\|S(\sum_{n=0}^N\sum_{i=1}^{2^n}
a_{n,i}h_{n,i})\|_p^p.
\end{array}
\end{equation*}
For the other direction we use Lemma
\ref{lm:eq_with_max_function}, (\ref{eq:M}), and the fact that
$\Lambda(F_{n,i})\ge M$ on $G_{n,i}=\varphi^{-1}(F_{n,i})$.
\begin{equation*}
\begin{array}{rl}
\int(\sum_{j=1}^{2^N}\Lambda(F_{N,j})\sum_{(n,i);h_{n,i}\supseteq
h_{N,j}}a_{n,i}^2)^{p/2}
&\ge \int(\sum_{j=1}^{2^N}{\bf
1}_{G_{N,j}}M\sum_{(n,i);h_{n,i}\supseteq
h_{N,j}}a_{n,i}^2)^{p/2}\\
&=\sum_{j=1}^{2^N}\int_{G_{N,j}}M^{p/2}(\sum_{(n,i);h_{n,i}\supseteq
h_{N,j}}a_{n,i}^2)^{p/2}\\
&\ge (\prod_{n=1}^N(1-\e_n))2^{-N}\int
M^{p/2}\sum_{j=1}^{2^N}
(\sum_{(n,i);h_{n,i}\supseteq h_{N,j}}a_{n,i}^2)^{p/2}\\
&\ge
(\prod_{n=1}^N(1-\e_n))(1+\e)^{\frac{-p}{2-p}}c\mu(A_0)
\|S(\sum_{n=0}^N\sum_{i=1}^{2^n} a_{n,i}h_{n,i})\|_p^p.
\end{array}
\end{equation*}
It is clear that the claim follows with $a = (\prod_{n=1}^{\infty}(1-\e_n))(1+\e)^{\frac{-p}{2-p}}c\mu(A_0)$.
\end{proof}

Now we continue as in the proof of \cite[Theorem 9.1, case $1<p<2$]{JMST}.
From the fact that the $|\|\cdot\||_p$ is equivalent to the usual norm in $\LP$, $\strictpbounds$, it is enough to prove only the ``moreover'' part of Theorem \ref{thm:stabilization}. The fact that $T(X)$ is $K_p$ complemented in $L_p$ will follow from \cite{JMST} (The norm of the projection there depends only on the isomorphism constant and on $p$, see Lemma 9.6 and the proof of Theorem 9.1 in the case $1<p<2$ there).

Let $\{\beta_{m,j}\}_{m=0,}^\infty\phantom{}_{j=1}^{2^m}$ be a sequence of positive numbers such that $\sum_{n,i}\beta_{n,i} = \delta$.
Since we obtained $\Lambda$ as a limit of successive convex combinations of $\{v_n^2(\cdot)\}$, there exists a sequence of disjoint finite sets $\{\sigma_{m,j}\}_{m=0,}^\infty\phantom{}_{j=1}^{2^m}\subseteq \mathbb{N}$ with
$\sigma_{m,j}>\inf\{l\,:\, F_{m,j}\in\AlgE_l\}$ and a sequence of non-negative numbers $\{\alpha_n\}_{n=1}^{\infty}$ such that $\sum_{n\in\sigma_{m,j}}\alpha_n = 1$, $m=0,1,\ldots$, $j=1,2,\ldots,2^m$, and
$$
\int \left |\sum_{n\in\sigma_{m,j}}\alpha_nv_n^2(F_{m,j}) - \Lambda(F_{m,j}) \right|^{p/2} < \beta_{m,j}\int \Lambda(F_{m,j})^{p/2}
$$
for all $m=0,1,\ldots$, $j=1,2,\ldots,2^m$. Put  $u_{m,j} = \sum_{n\in\sigma_{m,j}}\alpha_nv_n^2$.

As in  \cite[Theorem 9.1]{JMST}, we define a Gaussian Haar system by
$$
k_{m,j} = \sum_{n\in\sigma_{m,j}} \alpha_n^{1/2}\sum_{h_{n,i}\subseteq F_{m,j}} Th_{n,i}
$$
for all $m=0,1,\ldots$, $j=1,2,\ldots,2^m$.  Set
$X = \overline{\vspan}\{ \sum_{n\in\sigma_{m,j}} \alpha_n^{1/2}\sum_{h_{n,i}\subseteq F_{m,j}} h_{n,i}\}_{m=0,}^{\infty}\phantom{}_{j=1}^{2^m}$
and $Y = \overline{\vspan }\{k_{m,j}\}_{m=0,}^{\infty}\phantom{}_{j=1}^{2^m}$.
We first show that some multiple of the sequence $\{k_{m,j}\}$ is almost isometrically equivalent to the Haar basis in the norm $\||\cdot|\|_p$.

For all coefficients $\{a_{n,i}\}_{n=0}^N\phantom{}_{i=1}^{2^n}$ we have
\begin{equation}
\begin{array}{rl}
\|\sum_{n=0}^N\sum_{i=1}^{2^n} a_{n,i}^2(\Lambda(F_{n,i})-u_{n,i}(F_{n,i}))\|_{p/2}^{p/2} &= \int(\sum_{n=0}^N\sum_{i=1}^{2^n} a_{n,i}^2(\Lambda(F_{n,i})-u_{n,i}(F_{n,i})))^{p/2}\\
&\leq \int \sum_{n=0}^N\sum_{i=1}^{2^n} a_{n,i}^p|\Lambda(F_{n,i})-u_{n,i}(F_{n,i})|^{p/2}\\
&\leq \int \sum_{n=0}^N\sum_{i=1}^{2^n} \beta_{n,i}a_{n,i}^p|\Lambda(F_{n,i})|^{p/2}\\
&\leq \delta\int (\sum_{n=0}^N\sum_{i=1}^{2^n} a_{n,i}^2\Lambda(F_{n,i}))^{p/2}\\
&=\delta \|\sum_{n=0}^N\sum_{i=1}^{2^n} a_{n,i}^2\Lambda(F_{n,i})\|_{p/2}^{p/2}
\end{array}
\end{equation}
and using Claim \ref{claim:main} we immediately get
\begin{equation}\label{eq:equivbasis}
a(1-\delta)\|S(\sum_{n=0}^N\sum_{i=1}^{2^n} a_{n,i}h_{n,i})\|_p^p\le
\|\sum_{n=0}^N\sum_{i=1}^{2^n}
a_{n,i}^2u_{n,i}(F_{n,i})\|_{p/2}^{p/2}\le b(1+\delta)
\|S(\sum_{n=0}^N\sum_{i=1}^{2^n} a_{n,i}h_{n,i})\|_p^p.
\end{equation}
Since $\{Th_{n,i}\}$ are disjointly supported with respect to the Haar basis,  it follows that
$$
S^2(k_{m,j}) = \sum_{n\in\sigma_{m,j}} \alpha_nS^2(\sum_{h_{n,i}\subseteq F_{m,j}} Th_{n,i}) = \sum_{n\in\sigma_{m,j}} \alpha_n v^2_n(F_{m,j}) = u_{m,j}(F_{m,j})
$$
and now using the fact that $\{k_{n,i}\}$ are disjointly supported with respect to the Haar basis we get
\begin{equation}\label{eq:equivbasis2}
S^2(\sum_{n=0}^N\sum_{i=1}^{2^n} a_{n,i}k_{n,i}) = \sum_{n=0}^N\sum_{i=1}^{2^n} a_{n,i}^2S^2(k_{n,i}) = \sum_{n=0}^N\sum_{i=1}^{2^n} a_{n,i}^2 u_{n,i}(F_{n,i}).
\end{equation}
Now we just have to observe that for any $x\in\LP$ we have $\|S(x)\|_p^p = \|S^2(x)\|_{p/2}^{p/2}$ and combining this with (\ref{eq:equivbasis}) and (\ref{eq:equivbasis2}) gives us
\begin{equation}\label{eq:almostisometry}
a(1-\delta)\|S(\sum_{n=0}^N\sum_{i=1}^{2^n} a_{n,i}h_{n,i})\|_p^p\le
\|S(\sum_{n=0}^N\sum_{i=1}^{2^n} a_{n,i}k_{n,i})\|_p^p \le b(1+\delta)
\|S(\sum_{n=0}^N\sum_{i=1}^{2^n} a_{n,i}h_{n,i})\|_p^p.
\end{equation}
The last estimate shows that some multiple of the sequence $\{k_{n,i}\}$ is almost isomterically equivalent to the Haar basis with respect to $\||\cdot|\|_p$.
We must mention that (\ref{eq:almostisometry}) also implies that some multiple of $T$ is almost an isometry on $X$. This follows from the fact
$S^2(\sum_{n\in\sigma_{m,j}} \alpha_n^{1/2}\sum_{h_{n,i}\subseteq F_{m,j}} h_{n,i}) = {\bf 1}_{F_{m,j}}$, hence
$$
\mu(A_0)\|S(\sum_{n=0}^N\sum_{i=1}^{2^n} a_{n,i}h_{n,i})\|_p^p =
\|S(\sum_{n=0}^N\sum_{i=1}^{2^n} a_{n,i}(\sum_{m\in\sigma_{n,i}} \alpha_m^{1/2}\sum_{h_{m,j}\subseteq F_{n,i}} h_{m,j}))\|_p^p.
$$
\end{proof}

\begin{tabular}{lll}
D. Dosev&W.B. Johnson&G. Schechtman\\
Department of Mathematics&Department of Mathematics&Department of Mathematics\\
Weizmann Institute of Science&Texas A\&M University&Weizmann Institute of Science\\
Rehovot, Israel&College Station, TX  77843 U.S.A.&Rehovot, Israel\\
{\tt dosevd@weizmann.ac.il}&{\tt johnson@math.tamu.edu}
&{\tt gideon@weizmann.ac.il}\\
\end{tabular}

\end{document}